\documentclass[12pt]{amsart}
\usepackage{inputenc, amsmath,amssymb,graphicx, verbatim, bbm}
\usepackage[hidelinks]{hyperref}

\usepackage{amsthm,xcolor}

\usepackage[maxbibnames=99, backend=biber, style=alphabetic,]{biblatex}

\makeatletter
\@namedef{subjclassname@2020}{%
  \textup{2020} Mathematics Subject Classification}
\makeatother

\newtheorem{theorem}[equation]{Theorem}
\newtheorem{lemma}[equation]{Lemma}
\newtheorem{proposition}[equation]{Proposition}

\theoremstyle{definition}
\newtheorem{definition}[equation]{Definition}

\theoremstyle{remark}

\numberwithin{equation}{subsection}

\begin{document}
\title{Sárközy's theorem for shifted primes with restricted digits}

\author[Burgin]{Alex Burgin}
    \address{School of Mathematics, Georgia Institute of Technology, Atlanta GA 30332, USA}
    \email{aburgin6@gatech.edu}

 \begin{abstract}
    We study recurrence along shifted primes with restricted digits. By constructing a local approximant to the associated exponential sums, we prove the van der Corput property for shifted primes with restricted digits. This in turn shows that if $A\subset \mathbb{N}$ has positive upper Banach density, then there exists some prime $p$ with restricted digits and two elements $a_1,a_2\in A$ such that \begin{align*}
        a_1+p-1=a_2.
    \end{align*}
\end{abstract}

\subjclass[2020]{Primary 11A63, 11N05; Secondary 11N13, 11L20, 11B30, 37A44.}

\keywords{Restricted-digit primes, exponential sums over primes,
primes in arithmetic progressions, van der Corput sets,
intersective sets, Sárközy's theorem.}

\maketitle 

\vspace{-1cm}

\tableofcontents

\section{Introduction} 

Let $b\geq 2$ and $\mathcal{A}\subset \{0,\hdots,b-1\}$ be a set of allowed digits. The associated restricted digit set \begin{align*}
    \mathcal{C}=\bigcup_{N\in \mathbb{N}}\Big\{\sum_{j=0}^{N-1}a_j b^j:a_j\in \mathcal{A}\Big\}
\end{align*} contains $|\mathcal{A}|^N=X^\delta$ integers below $X=b^N$, where $\delta = \frac{\log |\mathcal{A}|}{\log b}$. Such examples form a natural family of sparse subsets of the integers, and one expects---excluding local congruence obstructions---approximately $X^\delta / \log X$ primes from $\mathcal{C}$ in $\{1,\hdots,X\}$. Even without considering the sparsity of the set, validating this prediction and addressing related questions are difficult, because primality is governed by multiplicative phenomena, whereas $\mathcal{C}$ has very rich additive structure.\\

A substantial literature studies the arithmetic and distribution of digit-restricted integers and primes. The distribution of missing-digit integers in residue classes was studied by Erd\H{o}s, Mauduit, and S\'{a}rk\"{o}zy \cite{EMS1,EMS2}, and subsequently by Konyagin \cite{konyagin}, while Dartyge and Mauduit \cite{semiprimes} proved the existence of semiprimes in such sets. Bourgain \cite{bourgainPrimes1,bourgainPrimes2} and Swaenepoel \cite{swaenepoel} studied primes with many prescribed digits. Burgin, Fragkos, Lacey, Mena, and Reguera established various recurrence properties for missing-digit sets in \cite{burgin2026integercantorsetsarithmetic,burgin2026szemeredistheoremcantorsets}, as did Saavedra-Araya in \cite{saavedraAraya}. Quite recently, Brokering Pinilla, Iosevich, and Krause \cite{brokeringIosevichKrause} obtained pointwise almost-everywhere convergence and variational estimates for ergodic averages along these sets. For primes with globally forbidden digits, Maynard proved in two papers \cite{maynardDecimal,maynardAsymptotic} that decimal primes may omit any fixed digit, and in the large-base regime obtained asymptotic formulas for primes and polynomials with restricted digits. Nath \cite{nathBV} later established Bombieri-Vinogradov-type distribution theorems for primes with one missing digit, and Leng and Sawhney \cite{lengSawhney} proved a restricted digit version of Vinogradov's three-primes theorem.\\

The motivating question of the current paper comes from additive recurrence. S\'{a}rk\"{o}zy \cite{sarkozyOriginal} proved, in a quantitative manner, that the shifted primes $\mathbb{P}-1$ form a set of recurrence. Kamae and Mendès France \cite{kamaeMendesFrance} showed that the shifted primes are a van der Corput set, which is stronger than recurrence; see also Slijepčević \cite{primesQuant} for preliminary quantitative bounds. Green \cite{greenPowerSavings} showed subsequently that one has a power savings in the quantitative van der Corput statement. We ask whether this recurrence phenomenon survives after the primes are thinned by a rigid digit restriction. This is not a formal consequence of the corresponding statement for either $\mathbb{P}-1$ or $\mathcal{C}$, since recurrence and the van der Corput property are not inherited by arbitrary subsets.\\

Our results are complementary to those of Nath and
Leng--Sawhney.  Nath proves substantially stronger
distribution estimates with respect to the size of the modulus: for primes
omitting one digit in a sufficiently large base, he obtains
Bombieri--Vinogradov-type estimates averaged over moduli growing as a power of the main parameter, together with stronger variants for suitably structured
weights.  By contrast, our Dirichlet-type theorem concerns fixed moduli.
However, it applies to a broader class of digit sets, allows moduli having
arbitrary common factors with the base, and determines the resulting
nonuniform local density explicitly.  In particular, the main term separates
the usual prime-to-$b$ congruence obstruction from a finite $b$-adic digit
density.  Leng and Sawhney pursue a different additive question: they prove
that every sufficiently large odd integer is a sum of three primes omitting a
fixed digit.  Their argument therefore establishes the averaged Fourier
control needed for a ternary representation theorem, whereas our second main
result gives cancellation of the individual restricted digit prime exponential sum
at every fixed irrational frequency.  Neither the progression estimates of
Nath nor the ternary theorem of Leng--Sawhney directly yields this pointwise
irrational-frequency statement or the resulting conclusion that the shifted
restricted digit primes form a van der Corput set.  Thus the novelty here lies
not in improving the available level of distribution, but in combining an
explicit fixed-modulus local theory with irrational Fourier decay, for a more
general class of digit sets, in a form adapted to recurrence.

\subsection{Standing conditions on $\mathcal{C}$}\label{subsec:standingAssumptions} Fix $\epsilon>0$. Throughout this paper, $b$ is taken sufficiently large in terms of $\epsilon$. We will consider restricted digit sets $\mathcal{C}$ with digits $\mathcal{A}\subset \{0,...,b-1\}$, where $\mathcal{A}$ satisfies the following conditions: \begin{enumerate}
    \item[(1)] $\{0,1\}\subset \mathcal{A}$
    \item[(2)] $\mathcal{A}$ is a disjoint union of $k$ intervals
    \item[(3)] $|\mathcal{A}|=b-s\geq kb^{4/5+\epsilon}$.
\end{enumerate}  

The first criterion eliminates local obstructions from powers of $b$, while the second and third provide sufficient Fourier control and ensure that $\mathcal{C}$ is not too sparse.\\

\subsection{The main results}

\begin{theorem}\label{thm:Dirichlet}
    Assume $\mathcal{C}$ satisfies the standing conditions above. Then, for each fixed modulus $m\geq 1$ and residue $t\in \mathbb{Z}/m\mathbb{Z}$, and any $K>0$, \begin{align*}
        \sum_{\substack{0\leq n<b^N\\ n\equiv t\ (\text{mod }m)}}\mathbf{1}_\mathcal{C}(n)\Lambda(n)=\kappa_{m,t}(b-s)^N+O_K\Big(\frac{(b-s)^N}{(\log b^N)^K}\Big)
    \end{align*} where \begin{align*}
       \kappa_{m,t}:=\frac{b}{\phi(bv)}\mathbf{1}((t,v)=1)\cdot (b-s)^{-L}\sum_{\substack{n<b^L\\ n\equiv t\ (\text{mod }u)}}\mathbf{1}_\mathcal{C}(n)\mathbf{1}((n,b)=1).
    \end{align*} Here we write $m=uv$, such that $(v,b)=1$ and $p|u\implies p|b$. $L\in \mathbb{N}$ is any sufficiently large number such that $u|b^{L}$ (the normalized sum is independent of $L$ in such a case: see Lemma \ref{lem:wellDefined}).
\end{theorem} By considering $m=1$, this recovers the asymptotic of Maynard \cite{maynardAsymptotic} for the number of primes in $\mathcal{C}$, as well as the error term. It also incorporates local (congruence and digit-based) obstructions to well-distribution in $m\mathbb{Z}+t$: \begin{enumerate}
    \item[(a)] If $(t,m)>1$, then it is an easy exercise to show that $\kappa_{m,t}=0$. 
    \item[(b)] If $b^j|m$ and $t-b^j\lfloor t/b^j\rfloor\not\in \mathcal{C}$, then one can also show that $\kappa_{m,t}=0$.
\end{enumerate}

Our second main result is a Vinogradov-type equidistribution theorem for exponential sums over primes in $\mathcal{C}$ (henceforth denoted as $\mathbb{P}_\mathcal{C}$).

\begin{theorem}\label{thm:Vinogradov}
    Assume $\mathcal{C}$ satisfies the standing conditions. Then, for any $\theta\in \mathbb{R}\setminus \mathbb{Q}$, \begin{align*}
        \sum_{0\leq n<b^N}\mathbf{1}_\mathcal{C}(n)\Lambda(n)e(n\theta)=o((b-s)^N).
    \end{align*}
\end{theorem} This estimate, although qualitative, is sufficient for the transference and additive applications in the next subsection.

\subsection{A S\'ark\"ozy-type result}\label{subsec:Sarkozy} As an application of Theorems \ref{thm:Dirichlet} and \ref{thm:Vinogradov}, we obtain a S\'ark\"ozy-type result regarding shifted primes with restricted digits. Recall that a set $S\subset \mathbb{Z}$ is called \textit{intersective} (or a \textit{set of recurrence}) if every subset $A\subset \mathbb{Z}$ of positive upper Banach density contains two distinct elements whose difference lies in $S$. Equivalently, $S$ is a set of recurrence if for every measure-preserving system $(X,\mathcal{B},\mu,T)$ and every $E\in \mathcal{B}$ with $\mu(E)>0$, there exists $s\in S$ such that $\mu(E\cap T^{-s}E)>0$.

Problems of this type originate in work of S\'{a}rk\"{o}zy, who showed that the set of squares and, later, the set of shifted primes $\mathbb{P}-1$ are intersective \cite{sarkozySquares} \cite{sarkozyOriginal}. Subsequent work of Lucier \cite{lucier}, Ruzsa-Sanders \cite{ruzsaSanders}, Wang \cite{wang}, and most recently, Green \cite{greenPowerSavings}, greatly improved the quantitative bounds in the case of shifted primes, producing a power-savings result.

In the present paper, we can obtain the S\'ark\"ozy-type result for the shifted set $\mathbb{P}_\mathcal{C}-1$, where $\mathcal{C}$ satisfies the standing conditions in \S\ref{subsec:standingAssumptions}. Our main application is the following theorem.

\begin{theorem}\label{thm:mainThm}
    Let $A\subset \mathbb{N}$ be a set with positive upper Banach density. Suppose $\mathcal{C}$ satisfies the standing conditions. Then, there exists some prime $p\in \mathbb{P}_\mathcal{C}$ and two elements $a_1,a_2\in A$ such that \begin{align*}a_1+p-1=a_2.\end{align*}
\end{theorem} 

The conditions that 0 and 1 are admissible digits are necessary. Consider $A=b^2\mathbb{N}$. We have that $A-A=b^2\mathbb{Z}$, and so if there was some $p\in \mathbb{P}_\mathcal{C}$ with $p-1\in A-A$, one would have that $p\equiv 1\ (\text{mod }b^2)$. Consequently, the last two digits of $p$ in base $b$ are 01, and so 0 and 1 must be allowed digits for $\mathbb{P}_\mathcal{C}-1$ to be intersective.

Our methods show that $\mathbb{P}_{\mathcal C} - 1$ is not merely intersective but a van der Corput set, a strictly stronger notion of recurrence. Green's aforementioned result arose from a quantitative argument involving the van der Corput property for ordinary primes. The deduction of these recurrence statements from Theorems \ref{thm:Dirichlet} and \ref{thm:Vinogradov} uses a spectral argument, which we detail in \S\ref{sec:recurrence}.

An alternative route to Theorem \ref{thm:mainThm} would be to show that, for each $r\geq 1$, the shifted restricted-digit primes $\mathbb{P}_\mathcal{C}-1$ contain a finite IP set of dimension $r$, which is to say, that there exist $n_1,\hdots,n_r$ such that \begin{align*}
    \text{FS}(n_1,\hdots,n_r)\subseteq \mathbb{P}_\mathcal{C}-1,
\end{align*} where \begin{align*}
    \text{FS}(n_1,\hdots, n_r):=\Big\{\sum_{i\in I}n_i:I\subseteq [r], I\neq \varnothing\Big\}.
\end{align*} The containment of arbitrarily large IP sets implies recurrence by a standard finite IP-recurrence argument (see, e.g., \cite{furstenbergKatznelson}). For the unrestricted primes, the corresponding assertion follows from the Green–Tao–Ziegler \cite{greenTao,greenTaoZiegler,greenTaoMobius} theory of finite-complexity linear systems in the primes, applied to linear forms of this type. No comparable linear-forms theorem is presently available for primes with restricted digits. Proving one would require higher-order pseudorandomness/Gowers norm estimates for $\mathbb{P}_\mathcal{C}$ beyond the Fourier estimates established here (i.e. $U^k$-estimates with $k\geq 3$).

\subsection{Notation} Let $\mathbb{P}_\mathcal{C}$ denote the set of primes in the restricted-digit set $\mathcal{C}$, so that $\mathbb{P}_\mathcal{C}:=\mathbb{P}\cap \mathcal{C}$. We let $e(x):=e^{2\pi ix}$ denote the standard complex exponential function. For a positive integer $X$, and a function $f:\mathbb{Z}_{\geq 0}\rightarrow \mathbb{C}$, we write \begin{align*}
    \widehat{f}_X(\theta):=\sum_{0\leq n<X}f(n)e(\theta n).
\end{align*} In the case $f=\mathbf{1}_\mathcal{C}$, we simply write \begin{align*}
    \widehat{C}_X(\theta):=\sum_{0\leq n<X}\mathbf{1}_\mathcal{C}(n)e(\theta n)
\end{align*} for convenience. We also write $\|\cdot\|$ to denote the distance to the nearest integer: it is easy to see that $\|x\|$ is comparable to $|e(x)-1|$. Finally, we use the standard asymptotic notation: for $f:\mathbb{R}\rightarrow \mathbb{C}$ and $g:\mathbb{R}\rightarrow \mathbb{R}^+$, $f\ll g$ or $f=O(g)$ means there exists some absolute constant $C>0$ such that $|f(x)|\leq Cg(x)$. Similarly, $f\ll_A g$ or $f=O_A(g)$ mean that there exists some constant $C=C(A)>0$ depending on a parameter $A$ such that $|f(x)|\leq C|g(x)|$. Since we are taking the restricted digit set $\mathcal{C}$ to be fixed throughout this paper, we will view $b$, $s$, and $k$ as absolute constants and drop them from any subscripts.

\section{The van der Corput property for Shifted Primes with Restricted Digits}\label{sec:recurrence}

Here we use Theorems \ref{thm:Dirichlet} and \ref{thm:Vinogradov} to show Theorem \ref{thm:mainThm}. As in \cite{greenPowerSavings}, one must construct a trigonometric polynomial with spectrum supported on the shifted primes; in our setting, one must have spectrum supported on shifted primes with restricted digits. Let us record the precise framework necessary, as given by Kamae and Mendès France in \cite{kamaeMendesFrance}. The introduction of Green's paper \cite{greenPowerSavings} is also a good reference for this section.

\begin{definition}
    A set $H\subset \mathbb{N}$ is a \textit{van der Corput} set if, whenever a sequence $(x_n)\subset \mathbb{T}$ has the property that \begin{align*}
        (x_{n+h}-x_n)_{n\in \mathbb{N}}
    \end{align*} is uniformly distributed modulo 1 for every $h\in H$, the sequence $(x_n)$ itself is uniformly distributed modulo 1.
\end{definition}

It is well-known that the van der Corput property implies intersectivity; it is actually strictly stronger than intersectivity, as shown by Bourgain \cite{bourgainVDC}. We refer to Green's paper, \S1.1 and \S1.2, for a standard proof of the implication.\\

Let $P(H)$ be the set of trigonometric polynomials of the form $f(x)=\sum_{h\in H}a_he(hx)$ with $(a_h)$ having finite support, such that $f(0)=1$ and $|f|\leq 1$. Let $P^*(H)$ be the closure of $P(H)$ under pointwise convergence. For $\eta\geq 0$, let \begin{align*}
    P_\eta^*(H):=\{f\in P^*(H):\text{Re }f(x)\geq -\eta\text{ for every }x\in \mathbb{T}\}.
\end{align*} Kamae and Mendès France showed the following: \begin{lemma}[Kamae–Mendès France, \cite{kamaeMendesFrance}]\label{lem:KMFcriterion}
   If $P_\eta^*(H)\neq \varnothing $ for every $\eta>0$, then $H$ is a van der Corput set.
\end{lemma}

We discuss the construction of trigonometric polynomials for the setting $H=\mathbb{P}_\mathcal{C}-1$. Begin by fixing $q\in \mathbb{N}$, and define \begin{align*}
    f_{q,N}(\alpha)&:=\frac{1}{F_{q,N}}\sum_{\substack{p<b^N\\ p\in \mathbb{P}_\mathcal{C}\\ p\equiv 1\ (\text{mod }q!)}}(\log p)e((p-1)\alpha), \\ F_{q,N}&:= \sum_{\substack{p<b^N\\ p\in \mathbb{P}_\mathcal{C}\\ p\equiv 1\ (\text{mod }q!)}}\log p.
\end{align*}

\begin{lemma}\label{lem:weightedVM}
    One has that \begin{align*}
        F_{q,N}\sim \kappa_{q!,1}(b-s)^N,
    \end{align*} where $\kappa_{q!,1}>0$ is the constant in Theorem \ref{thm:Dirichlet}.
\end{lemma}

\begin{proof}
    It suffices to show that the proper prime power contribution of the van Mangoldt function is negligible compared to the main term here. Since \begin{align*}
        \sum_{k\geq 2}\sum_{\substack{p^k<b^N\\ p\in \mathbb{P}}}\log p\ll b^{N/2}\log^{O(1)}(b^N),
    \end{align*} this follows from the density hypothesis in the standing conditions.
\end{proof}

\begin{proposition}\label{prop:constructionMeasures}
    For every fixed $q\in \mathbb{N}$ and every $\alpha\in \mathbb{T}$, the limit \begin{align*}
    f_q(\alpha):=\lim_{N\rightarrow \infty} f_{q,N}(\alpha)
    \end{align*} exists. It has the following properties: \begin{enumerate}
        \item[(i)] $f_q(\alpha)=0$ whenever $\alpha$ is irrational
        \item[(ii)] if $\alpha=a/r$ is rational and $r|q!$, then $f_q(\alpha)=1$.
    \end{enumerate}
\end{proposition}

\begin{proof}
    We first note that, by the same logic as in Lemma \ref{lem:weightedVM}, one has that \begin{align}\label{eq:VMmeasure}
        f_{q,N}(\alpha)=\frac{1}{F_{q,N}}\sum_{\substack{n<b^N\\ n\equiv 1\ (\text{mod }q!)}}\mathbf{1}_\mathcal{C}(n)\Lambda(n)e((n-1)\alpha)+o(1).
    \end{align} Since \begin{align*}
        \mathbf{1}_{n\equiv 1\ (\text{mod }q!)}=\frac{1}{q!}\sum_{t\ (\text{mod }q!)}e\Big(\frac{t(n-1)}{q!}\Big)
    \end{align*} one may write \begin{align*}
        f_{q,N}(\alpha)=\frac{e(-\alpha)}{q!F_{q,N}}\sum_{t\ (\text{mod }q!)}e(-t/q!)\sum_{n<b^N}\mathbf{1}_\mathcal{C}(n)\Lambda(n)e\Big(n\Big(\alpha + \frac{t}{q!}\Big)\Big)+o(1).
    \end{align*} For $\alpha$ irrational, applying Lemma \ref{lem:weightedVM} and noting each $\alpha+t/q!$ is irrational, one may apply Theorem \ref{thm:Vinogradov} to obtain that $\lim_{N\rightarrow \infty}f_{q,N}(\alpha)=0$.

   Now consider $\alpha = a/r$ rational. Let $M=\text{lcm}(r,q!)$, then \begin{align*}
       f_{q,N}(\alpha)&=\frac{1}{F_{q,N}}\sum_{\substack{n<b^N\\ n\equiv 1\ (\text{mod }q!)}}\mathbf{1}_\mathcal{C}(n)\Lambda(n)e((n-1)a/r)+o(1) \\ &=\frac{1}{F_{q,N}}\sum_{c\ (\text{mod }M)}\sum_{\substack{n<b^N\\ n\equiv 1\ (\text{mod }q!)\\ n\equiv c\ (\text{mod }M)}}\mathbf{1}_\mathcal{C}(n)\Lambda(n)e((n-1)a/r)+o(1) \\ &= \sum_{\substack{c\ (\text{mod }M)\\ c\equiv 1\ (\text{mod }q!)}}e(-a/r)\cdot \frac{1}{F_{q,N}}\sum_{\substack{n<b^N\\ n\equiv c\ (\text{mod }M)}}\mathbf{1}_\mathcal{C}(n)\Lambda(n)e(na/r)+o(1) \\ &= \sum_{\substack{c\ (\text{mod }M)\\ c\equiv 1\ (\text{mod }q!)}}e((c-1)a/r)\cdot \frac{1}{F_{q,N}}\sum_{\substack{n<b^N\\ n\equiv c\ (\text{mod }M)}}\mathbf{1}_\mathcal{C}(n)\Lambda(n)+o(1).
   \end{align*} Applying Theorem \ref{thm:Dirichlet} and Lemma \ref{lem:weightedVM}, one obtains that \begin{align*}
       \lim_{N\rightarrow \infty}f_{q,N}(\alpha)=\sum_{\substack{c\ (\text{mod }M)\\ c\equiv 1\ (\text{mod }q!)}}e((c-1)a/r)\cdot \frac{\kappa_{M,c}}{\kappa_{q!,1}}.
   \end{align*} This proves the existence of the rational limit in general, and it is clear when $r|q!$, since $M=q!$ in such a case, that this limit is one.
\end{proof}

Each function $f_q$ lies in $P^*(H)$. Since the pointwise limit $\lim_{q\rightarrow \infty}f_q$ is $\mathbf{1}_{\mathbb{Q}/ \mathbb{Z}}$ by Proposition \ref{prop:constructionMeasures}, one then obtains that $\mathbf{1}_{\mathbb{Q}/\mathbb{Z}}\in P^*(H)$. Since $\mathbf{1}_{\mathbb{Q}/\mathbb{Z}}\geq 0$, we obtain in particular that $\mathbf{1}_{\mathbb{Q}/\mathbb{Z}}\in P_\eta^*(H)$ for each $\eta> 0$, and so by Lemma \ref{lem:KMFcriterion} one obtains that $H=\mathbb{P}_\mathcal{C}-1$ is a van der Corput set.

\section{Fourier Estimates}\label{sec:Fourier}

In the next two subsections, the bounds are similar to those of \cite{maynardAsymptotic}: we include for completeness and exposition. We assume the standing conditions in \S\ref{subsec:standingAssumptions} on $\mathcal{C}$ in each of these lemmas.

\subsection{$L^p$, Large Sieve, and Hybrid Estimates}

We begin with a couple estimates for $\widehat{\Lambda}_{x}(t)$, which are classical.

\begin{lemma}\label{lem:vonMangoldtFourierBounds}
    Let $\theta = a/d+\beta$ with $(a,d)=1$ and $|\beta|<1/d^2$. Then, if $\widehat{\Lambda}_X(\theta):=\sum_{0\leq n<X}\Lambda(n)e(n\theta)$, then both of the following bounds hold: \begin{enumerate}
        \item[(1)] The major arc bound (where $\beta d$ is small). One has that \begin{align*}
            |\widehat{\Lambda}_X(\theta)|\ll \Big(X^{4/5}+(dX)^{1/2}+\frac{X}{d^{1/2}}\Big)(\log X)^4
        \end{align*} See, e.g., \cite[Thm 13.6]{iwaniecKowalski}
        \item[(2)] The minor arc bound (where $\beta d$ is large). One has that \begin{align*}
            |\widehat{\Lambda}_X(\theta)|\ll \Big(X^{4/5}+\frac{X^{1/2}}{|d\beta|^{1/2}}+X|d\beta|^{1/2}\Big)(\log X)^{4}.
        \end{align*} See, e.g., \cite[Lemma 4.2]{maynardAsymptotic}
    \end{enumerate}
\end{lemma}

We also have various results regarding $\widehat{C}_{b^N}$. Recall that $\mathcal{C}$ consists of nonnegative integers that have only digits in $\mathcal{A}$ when written in base $b$, and the one standing condition is that the set of digits $\mathcal{A}$ is a union of $k$ intervals.

\begin{lemma}[$L^1$ Bound, cf. \cite{maynardAsymptotic} Lemma 5.1]\label{lem:L1estimate} If $C_0:=k+\frac{2(b-s)}{b\log b}$, then \begin{align}
    \sup_{x\in \mathbb{R}}\sum_{0\leq a< b^N}\Big|\widehat{C}_{b^{N}}\Big(x+\frac{a}{b^N}\Big)\Big|\leq (C_0b\log b)^N.
\end{align}
\end{lemma}

\begin{proof}
    We may write \begin{align*}
        \widehat{C}_{b^N}(x)=\prod_{i=0}^{N-1}\Big(\sum_{a\in \mathcal{A}}e(b^iax)\Big).
    \end{align*} Since $\mathcal{A}=\bigsqcup_{j=1}^kI_j$ for intervals $I_j$, we have by the triangle inequality that \begin{align*}
        \Big|\sum_{a\in \mathcal{A}}e(b^iax)\Big|\leq \sum_{j=1}^k\Big|\sum_{a\in I_j}e(b^iax)\Big|\leq \sum_{j=1}^k \frac{2}{|1-e(b^ix)|}\leq \frac{k}{2\|b^ix\|},
    \end{align*} where we summed the partial geometric series over each interval. Since $|\mathcal{A}|=b-s$, we then have the bound \begin{align}\label{eq:directL-inftyBound}
        |\widehat{C}_{b^N}(x)|\leq \prod_{i=0}^{N-1}\min\{b-s,\frac{k}{2\|b^ix\|}\}.
    \end{align} For $x\in [0,1)$, we may write $x=\sum_{i=1}^Nx_ib^{-i}+\epsilon$, where $x_1,...,x_N\in \{0,...,b-1\}$ and $\epsilon\in [0,b^{-N})$. Thus $\|b^ix\|^{-1}=\|x_{i+1}/b+\epsilon_i\|^{-1}$ for $\epsilon_i\in [0,b^{-1})$. We may then bound \begin{align*}
        \|b^ix\|^{-1}=\|x_{i+1}/b+\epsilon_i\|^{-1}\leq \max\Big\{\frac{b}{x_{i+1}},\frac{b}{b-1-x_{i+1}}\Big\},
    \end{align*} to provide that \begin{align*}
        |\widehat{C}_{b^N}(x)|\leq \prod_{i=0}^{N-1}\min\Big\{b-s,\frac{kb}{2}\max\Big\{\frac{1}{x_{i+1}},\frac{1}{b-1-x_{i+1}}\Big\}\Big\}.
    \end{align*} Let $S_x:=(x+a/b^N\ (\text{mod }1))_{0\leq a< b^N}$. For any $t\neq t'\in S_x$, by writing $t=\sum_{i=1}^Nt_ib^{-i}+\epsilon$ and $t'=\sum_{i=1}^N t_i'b^{-i}+\epsilon'$ as above, we claim that $(t_1,...,t_N)\neq (t_1',...,t_N')$. Indeed, if we had equality, we would have $\|t-t'\|<b^{-N}$, a contradiction as $t$ and $t'$ are separated by at least $b^{-N}$ (mod 1). So, \begin{align*}
        \sum_{0\leq a<b^N}\Big|\widehat{C}_{b^N}\Big(x+\frac{a}{b^N}\Big)\Big|&=\sum_{t\in S_x}|\widehat{C}_{b^N}(t)| \\ &\leq \sum_{0\leq t_1,...,t_N<b}\prod_{i=0}^{N-1}\min\Big\{b-s,\frac{kb}{2}\max\Big\{\frac{1}{t_{i+1}},\frac{1}{b-1-t_{i+1}}\Big\}\Big\} \\ &= \prod_{i=0}^{N-1}\sum_{0\leq t<b}\min\Big\{b-s,\frac{kb}{2}\max\Big\{\frac{1}{t},\frac{1}{b-1-t}\Big\}\Big\}.
    \end{align*} Since $1/t>\frac{1}{b-1-t}$ precisely when $t<\frac{b-1}{2}$ we may compute \begin{align*}
        &\sum_{0\leq t<b}\min\Big\{b-s,\frac{kb}{2}\max\Big\{\frac{1}{t},\frac{1}{b-1-t}\Big\}\Big\}\\ &=\sum_{0\leq t<\frac{b-1}{2}}\min\{b-s,\frac{kb}{2t}\}+\sum_{\frac{b-1}{2}\leq t<b}\min\{b-s,\frac{kb}{2(b-1-t)}\} \\ &\leq  2(b-s)+2\sum_{1\leq t<\frac{b-1}{2}}\frac{kb}{2t}\\ &=2(b-s)+kb\sum_{1\leq t<\frac{b-1}{2}}1/t\\ &\leq 2(b-s)+kb\log b.
    \end{align*} Consequently, \begin{align}
        \sum_{0\leq a< b^N}\Big|\widehat{C}_{b^N}\Big(x+\frac{a}{b^N}\Big)\Big|\leq \Big(2(b-s)+kb\log b\Big)^N.
    \end{align} Since $x$ was arbitrary, this provides the desired result.
\end{proof}

\begin{lemma}[Large Sieve Estimate, cf. \cite{maynardAsymptotic} Lemma 5.2]\label{lem:largeSieveEstimate} Let $C_0$ be as in Lemma \ref{lem:L1estimate}. Then, \begin{align*}
    \sup_{x\in \mathbb{R}}\sum_{d\sim D}\sum_{\substack{0\leq \ell<d\\ (\ell,d)=1}}\sup_{|\epsilon|<\frac{1}{10D^2}}\Big|\widehat{C}_{b^N}\Big(\frac{\ell}{d}+x+\epsilon\Big)\Big|\ll (D^2+b^N)(C_0\log b)^N.
\end{align*}
    
\end{lemma}

\begin{proof}
   By the fundamental theorem of calculus, for any $u\in \mathbb{R}$ we have $\widehat{C}_{b^N}(x)=\widehat{C}_{b^N}(u)+\int_u^x\widehat{C}_{b^N}'(v)dv$. Averaging this over $u\in [x-\delta,x+\delta]$ and applying the triangle inequality, we deduce that \begin{align}
       |\widehat{C}_{b^N}(x)|\ll \frac{1}{\delta}\int_{x-\delta}^{x+\delta}|\widehat{C}_{b^N}(v)|dv+\int_{x-\delta}^{x+\delta}|\widehat{C}_{b^N}'(v)|dv.
   \end{align} After selecting one $\epsilon_{d,\ell}$ for each supremum, the perturbed points $\ell/d+x+\epsilon_{d,\ell}$ are separated from one another by $\gg 1/|D|^2$. Choosing $\delta\approx 1/|D|^2$ then provides (by disjointness of these small intervals of integration) that \begin{align}\label{eq:sobolevIneqType}
       \sum_{d\sim D}\sum_{\substack{0\leq \ell<d\\ (\ell,d)=1}}\sup_{|\epsilon|<\frac{1}{10D^2}}\Big|\widehat{C}_{b^N}\Big(\frac{\ell}{d}+x+\epsilon\Big)\Big|\ll D^2\int_0^1|\widehat{C}_{b^N}(v)|dv+\int_0^1|\widehat{C}_{b^N}'(v)|dv.
   \end{align}

Using the product rule for derivatives, we may write \begin{align*}
    \widehat{C}_{b^N}'(v)&=\Big(\prod_{j=0}^{N-1}\sum_{0\leq d<b}\mathbf{1}_\mathcal{C}(d)e(db^jv)\Big)'\\ &=2\pi i\sum_{j=0}^{N-1}b^j\Big(\sum_{0\leq d<b}d\mathbf{1}_\mathcal{C}(d)e(db^jv)\Big)\prod_{\substack{0\leq i<N\\ i\neq j}}\sum_{0\leq d<b}\mathbf{1}_\mathcal{C}(d)e(db^iv),
\end{align*} and so \begin{align*}
    |\widehat{C}_{b^N}'(v)|\ll \sum_{j=0}^{N-1}b^{j+1}\prod_{\substack{0\leq i<N\\ i\neq j}}\min\{b-s,\frac{k}{2\|b^iv\|}\}\ll b^N\prod_{0\leq i<N}\min\{b-s,\frac{k}{2\|b^iv\|}\}.
\end{align*} If we write $v=\sum_{i=1}^Nv_ib^{-i}+\epsilon$ with $v_i\in \{0,...,b-1\}$ and $\epsilon\in [0,b^{-N})$ we see by another averaging argument that \begin{align*}
    \int_0^1\prod_{0\leq i<N}\min\{b-s,\frac{k}{2\|b^iv\|}\}dv&=b^{-N}\int_0^1\sum_{j=0}^{b^N-1}\prod_{0\leq i<N}\min\{b-s,\frac{k}{2\|b^i(v+j/b^N)\|}\}dv \\ &\leq b^{-N}\int_0^1 (C_0b\log b)^Ndv \\ &= (C_0\log b)^N,
\end{align*} where for each $0\leq v\leq 1$ we use the bounds arising from digit considerations within the proof of Lemma \ref{lem:L1estimate}. From this we deduce that \begin{align*}
    \int_0^1|\widehat{C}_{b^N}'(v)|dv&\leq (C_0b\log b)^N \\
    \int_0^1|\widehat{C}_{b^N}(v)|dv&\leq (C_0\log b)^N
\end{align*} and so from (\ref{eq:sobolevIneqType}) we have that  \begin{align*}
    \sum_{d\sim D}\sum_{\substack{0\leq \ell <d\\ (\ell,d)=1}}\sup_{|\epsilon|<\frac{1}{10D^2}}\Big|\widehat{C}_{b^N}\Big(\frac{\ell}{d}+x+\epsilon\Big)\Big|\ll (D^2+b^N)(C_0\log b)^N.
\end{align*}
\end{proof}

\begin{lemma}[Hybrid Estimate, cf. \cite{maynardAsymptotic} Lemma 5.3]\label{lem:hybridEstimate} Let $B,D\gg 1$, with $B<\frac{b^N}{10D^2}$, and $C_0$ be as in Lemma \ref{lem:L1estimate}. Set $\alpha:=\log(C_0\frac{b}{b-s}\log b)/\log b$. Then, \begin{align*}
    \sup_{x\in \mathbb{R}}\sum_{d\sim D}\sum_{\substack{\ell<d\\ (\ell,d)=1}}\sum_{\substack{|\eta|<B\\ b^N\ell/d+\eta\in \mathbb{Z}}}\Big|\widehat{C}_{b^N}\Big(x+\frac{\ell}{d}+\frac{\eta}{b^N}\Big)\Big|\ll_b (b-s)^N(D^2B)^{\alpha}.
\end{align*} 
    
\end{lemma}

\textbf{Remark}. The constant $\alpha$ is important for controlling the error terms in our asymptotics. A larger base $b$ and a denser set of digits $\mathcal{A}$ gives us a smaller value of $\alpha$. Under the standing conditions, which we are assuming here, we get precisely that $\alpha<1/5$, which gives us the desired control of these error terms.

\begin{proof}
    For any $n_1\in [0,N]$ and $y\in \mathbb{R}$ we have from the product structure of $\widehat{C}_{b^N}$ that \begin{align*}
        \widehat{C}_{b^N}(y)=\widehat{C}_{b^{N-n_1}}(y)\widehat{C}_{b^{n_1}}(b^{N-n_1}y)
    \end{align*} and so \begin{align*}
        \Big|\widehat{C}_{b^N}\Big(x+\frac{\ell}{d}+\frac{\eta}{b^N}\Big)\Big|&=\Big|\widehat{C}_{b^{N-n_1}}\Big(x+\frac{\ell}{d}+\frac{\eta}{b^N}\Big)\Big|\cdot \Big|\widehat{C}_{b^{n_1}}\Big(b^{N-n_1}x+\frac{b^{N-n_1}\ell}{d}+\frac{\eta}{b^{n_1}}\Big)\Big|. \end{align*} Another iteration yields $\widehat{C}_{b^{N-n_1}}(y)=\widehat{C}_{b^{n_2}}(y)\widehat{C}_{b^{N-n_1-n_2}}(b^{N-n_1-n_2}y)$, and so applying the trivial bound $|\widehat{C}_{b^{N-n_1-n_2}}(y)|\leq (b-s)^{N-n_1-n_2}$ we produce \begin{align*}
            &\Big|\widehat{C}_{b^N}\Big(x+\frac{\ell}{d}+\frac{\eta}{b^N}\Big)\Big|\\ &\leq (b-s)^{N-n_1-n_2}\Big|\widehat{C}_{b^{n_2}}\Big(x+\frac{\ell}{d}+\frac{\eta}{b^N}\Big)\Big|\cdot \Big|\widehat{C}_{b^{n_1}}\Big(b^{N-n_1}x+\frac{b^{N-n_1}\ell}{d}+\frac{\eta}{b^{n_1}}\Big)\Big| \\ &\leq (b-s)^{N-n_1-n_2}\Big|\widehat{C}_{b^{n_1}}\Big(b^{N-n_1}x+\frac{b^{N-n_1}\ell}{d}+\frac{\eta}{b^{n_1}}\Big)\Big|\sup_{|\epsilon|<Bb^{-N}}\Big|\widehat{C}_{b^{n_2}}\Big(x+\frac{\ell}{d}+\epsilon\Big)\Big|.
        \end{align*} Thus, \begin{align*}
            (\star)&:=\sum_{d\sim D}\sum_{\substack{\ell<d\\ (\ell,d)=1}}\sum_{\substack{|\eta|<B\\ b^N\ell/d+\eta\in \mathbb{Z}}}\Big|\widehat{C}_{b^N}\Big(x+\frac{\ell}{d}+\frac{\eta}{b^N}\Big)\Big|\\ &\leq  (b-s)^{N-n_1-n_2}\sum_{d\sim D}\sum_{\substack{\ell<d\\ (\ell,d)=1}}\sup_{|\epsilon|<Bb^{-N}}\Big|\widehat{C}_{b^{n_2}}\Big(x+\frac{\ell}{d}+\epsilon\Big)\Big|\\ &\times\sum_{\substack{|\eta|<B\\ b^N\ell/d+\eta\in \mathbb{Z}}}\Big|\widehat{C}_{b^{n_1}}\Big(b^{N-n_1}x+\frac{b^{N-n_1}\ell}{d}+\frac{\eta}{b^{n_1}}\Big)\Big|.
        \end{align*} Choose $n_1$ minimal such that $b^{n_1}>B$, and so \begin{align*}
            (\star)&\leq (b-s)^{N-n_1-n_2}\sum_{d\sim D}\sum_{\substack{\ell<d\\ (\ell,d)=1}}\sup_{|\epsilon|<Bb^{-N}}\Big|\widehat{C}_{b^{n_2}}\Big(x+\frac{\ell}{d}+\epsilon\Big)\Big|\\ &\times\sum_{\substack{|\eta|<b^{n_1}\\ b^N\ell/d+\eta\in \mathbb{Z}}}\Big|\widehat{C}_{b^{n_1}}\Big(b^{N-n_1}x+\frac{b^{N-n_1}\ell}{d}+\frac{\eta}{b^{n_1}}\Big)\Big|.
        \end{align*} Notice that $\frac{b^{N-n_1}\ell}{d}+\frac{\eta}{b^{n_1}}=b^{-n_1}\Big(\frac{b^N\ell}{d}+\eta\Big)=a/b^{n_1}$ for some unique $a\in \mathbb{Z}/b^{n_1}\mathbb{Z}$, and so the inner sum is majorized by the $L^1$ sum at scale $b^{n_1}$ (Lemma \ref{lem:L1estimate}). So, \begin{align*}
            (\star)\leq (b-s)^{N-n_1-n_2}(C_0b\log b)^{n_1}\sum_{d\sim D}\sum_{\substack{\ell<d\\ (\ell,d)=1}}\sup_{|\epsilon|<Bb^{-N}}\Big|\widehat{C}_{b^{n_2}}\Big(x+\frac{\ell}{d}+\epsilon\Big)\Big|
        \end{align*} Then, since $B<\frac{b^N}{10D^2}$, we may apply Lemma \ref{lem:largeSieveEstimate} to deduce that \begin{align*}
            (\star)\ll (b-s)^{N-n_1-n_2}(C_0b\log b)^{n_1}(D^2+b^{n_2})(C_0\log b)^{n_2}
        \end{align*} Choose $n_2=\min\{N-n_1,\lfloor \frac{2\log D}{\log b}\rfloor\}$. We must verify that \begin{align}\label{eq:parameter1}
            (b-s)^{-n_1-n_2}(C_0b\log b)^{n_1}D^2(C_0\log b)^{n_2}&\ll (D^2B)^\alpha \\ \label{eq:parameter2}(b-s)^{-n_1-n_2}(C_0b\log b)^{n_1}b^{n_2}(C_0\log b)^{n_2}&\ll (D^2 B)^\alpha.
        \end{align} To verify (\ref{eq:parameter1}), first consider the case where $n_2=N-n_1$. In this case, one may bound the lefthand side of the inequality as follows: \begin{align*}
            (b-s)^{-n_1-n_2}(C_0b\log b)^{n_1}D^2(C_0\log b)^{n_2}&= \Big(\frac{C_0 b\log b}{b-s}\Big)^{n_1+n_2}b^{-n_2}D^2 \\ &= \Big(\frac{C_0 b\log b}{b-s}\Big)^N b^{-n_2}D^2 \\ &= b^{\alpha N}b^{n_1-N}D^2 \\ &\ll D^2B b^{(\alpha -1)N}.
        \end{align*} Here, we used the definition of $\alpha$, alongside the choice of $n_1$ above (so that $b^{n_1}\asymp B$). So, to show (\ref{eq:parameter1}) in this case, it suffices to show that $D^2B b^{(\alpha-1)N}\ll (D^2 B)^\alpha$, i.e. $(D^2 B)^{1-\alpha}\ll (b^N)^{1-\alpha}$ but this is immediate, as $\alpha\in (0,1)$ and $D^2 B< b^N/10$ by assumption.

        Now we verify (\ref{eq:parameter1}) in the case where $n_2=\lfloor \frac{2\log D}{\log b}\rfloor$. In such a case, $b^{n_2}\asymp D^2$. We may compute \begin{align*}
            (b-s)^{-n_1-n_2}(C_0 b\log b)^{n_1}D^2(C_0\log b)^{n_2}&= \Big(\frac{C_0 b\log b}{b-s}\Big)^{n_1+n_2}b^{-n_2}D^2 \\ &\asymp \Big(\frac{C_0 b\log b}{b-s}\Big)^{n_1+n_2} \\ &= b^{\alpha(n_1+n_2)} \\ &\ll (D^2B)^\alpha,
        \end{align*} where here we used the definition of $\alpha$ alongside the choices of $n_1$ and $n_2$ (so that $b^{n_1}\asymp B$ and $b^{n_2}\asymp D^2$).\\

        Finally, we verify (\ref{eq:parameter2}). We have that \begin{align*}
            (b-s)^{-n_1-n_2}(C_0 b\log b)^{n_1}b^{n_2}(C_0\log b)^{n_2}&=\Big(\frac{C_0 b\log b}{b-s}\Big)^{n_1+n_2}\\ &= b^{\alpha(n_1+n_2)} \\ &\ll B^\alpha (b^{n_2})^\alpha.
        \end{align*} Since $b^{n_2}\ll D^2$ in either case of how we choose $n_2$, we produce the desired inequality.
\end{proof}

\begin{lemma}[$L^\infty$ Bound, cf. \cite{maynardAsymptotic} Lemma 5.4]\label{lem:l-inftyBound} Take $b\geq 4$. Let $1<d<b^{N/3}$ be an integer, and $\ell\in \mathbb{Z}$, such that $b^i\ell/d\not\in \mathbb{Z}$ for each $i\geq 1$, and let $|\epsilon|<(2b^{2N/3})^{-1}$. Then, we have \begin{align*}
    \Big|\widehat{C}_{b^N}\Big(\frac{\ell}{d}+\epsilon\Big)\Big|\leq (b-s)^{N}\exp(-cN/\log d)
\end{align*} for a constant $c>0$ depending only on $b$.
\end{lemma}

\begin{proof}
    We note that, for any $\theta\in \mathbb{T}$, \begin{align*}
        |1+e(\theta)|^2=2+2\cos(2\pi \theta)\leq 4\exp(-2\|\theta\|^2)
    \end{align*} and so, since the set of admissible digits $\mathcal{A}$ contains both 0 and 1, we have \begin{align*}\Big|\sum_{n\in \mathcal{A}}e(n\theta)\Big|\leq b-s-2+2\exp(-\|\theta\|^2)\leq (b-s)\exp(-\|\theta\|^2/b)\end{align*} This provides then that \begin{align*}
        |\widehat{C}_{b^N}(t)|=\prod_{i=0}^{N-1}\Big|\sum_{n\in \mathcal{A}}e(nb^it)\Big|\leq (b-s)^{N}\exp\Big(-\frac{1}{b}\sum_{i=0}^{N-1}\|b^it\|^2\Big).
    \end{align*} Now, if $\|b^it\|<1/2b$ then $\|b^{i+1}t\|=b\|b^it\|$. If $t=\ell/d$ and $db^i\ell/d\not\in \mathbb{Z}$ for each $i\geq 1$, then $\|b^it\|\geq  1/d$ for all $i$. Similarly, if $t=\ell/d+\epsilon$ with $\ell,d$ as before, $|\epsilon|<b^{-2N/3}/2$, and $d<b^{N/3}$, then for $i<N/3$ we have that $\|b^it\|\geq 1/d-b^i|\epsilon|\geq 1/2d$. By induction, one can show for each $i\geq 0$ and $J<N/3-i$ that either $\|b^{i+j}(\ell/d+\epsilon)\|>1/2b^2$ for some $0\leq j<J$, or $\|b^{i+J}(\ell/d+\epsilon)\|\geq b^J/2d$. Thus, we deduce that for any interval $I$ of size $\frac{\log d}{\log b}$ in $[0,N/3]$, there exists some $i\in I$ such that $\|b^i(\ell/d+\epsilon)\|\geq 1/2b^2$. This provides that \begin{align*}
        \sum_{i=0}^{N-1}\Big\|b^i\Big(\frac{\ell}{d}+\epsilon\Big)\Big\|^2\geq \frac{1}{4b^4}\Big\lfloor \frac{N\log b}{3\log d}\Big\rfloor\gg_b \frac{N}{\log d}. 
    \end{align*} So, \begin{align*}
        \Big|\widehat{C}_{b^N}\Big(\frac{\ell}{d}+\epsilon\Big)\Big|\leq (b-s)^N\exp(-cN/\log d)
    \end{align*} for a constant $c=c(b)>0$, to provide the result.
\end{proof}

\subsection{The Minor Arcs}\label{sec:MinorArcs}

We may use the previous estimates to efficiently control what will become our minor arcs.

\begin{lemma}[Minor Arcs, cf.  \cite{maynardAsymptotic} Lemma 6.1]\label{lem:minorArcs}
    Let $1\ll B\ll b^N/D_0D$ and $1\ll D\ll D_0\ll b^{N/2}$. Let $\theta\in \mathbb{T}$ be arbitrary. Then we have \begin{align*}
        \sum_{d\sim D}\sum_{\substack{0\leq \ell<d\\ (\ell,d)=1}}\sum_{\substack{|\eta|\sim B\\ b^N\ell/d+\eta\in \mathbb{Z}}}\Big|\widehat{C}_{b^N}\Big(\theta+\frac{\ell}{d}+\frac{\eta}{b^N}\Big)\widehat{\Lambda}_{b^N}\Big(-\frac{\ell}{d}-\frac{\eta}{b^N}\Big)\Big|\\ \ll_b N^4b^N(b-s)^N\Big(\frac{1}{(D^2B)^{\frac{1}{5}-\alpha}}+\frac{b^{\alpha N}}{D_0^{1/2}}\Big)
    \end{align*} and \begin{align*}
        \sum_{d\sim D}\sum_{\substack{0\leq \ell<d \\ (\ell,d)=1}}\sum_{\substack{|\eta|\ll 1\\ b^N\ell/d+\eta\in \mathbb{Z}}}\Big|\widehat{C}_{b^N}\Big(\theta+\frac{\ell}{d}+\frac{\eta}{b^N}\Big)\widehat{\Lambda}_{b^N}\Big(-\frac{\ell}{d}-\frac{\eta}{b^N}\Big)\Big|\\ \ll_b N^4b^N(b-s)^N\Big(\frac{1}{D^{\frac{1}{5}-\alpha}}+\frac{D^{2\alpha+\frac{1}{2}}}{b^{N/2}}\Big).
    \end{align*} Here $\alpha\in (0,\frac{1}{5})$ is the constant in Lemma \ref{lem:hybridEstimate}.
\end{lemma}

\begin{proof}
    Let $\Sigma_1$ denote the first set of sums, and $\Sigma_2$ the second. By Lemma \ref{lem:vonMangoldtFourierBounds}, we have that \begin{align*}
        \sup_{\substack{d\sim D\\ (\ell,d)=1 \\ |\eta|\sim B}}\Big|\widehat{\Lambda}_{b^N}\Big(-\frac{\ell}{d}-\frac{\eta}{b^N}\Big)\Big|\ll_b \Big(b^{\frac{4N}{5}}+\frac{b^N}{(DB)^{1/2}}+(DB)^{1/2}b^{N/2}\Big)N^4
    \end{align*} and, by Lemma \ref{lem:hybridEstimate}, since $B\ll \frac{b^N}{D^2}$, \begin{align*}
        \sum_{d\sim D}\sum_{\substack{0\leq \ell<d\\ (\ell,d)=1}}\sum_{\substack{|\eta|\sim B\\ b^N\ell/d+\eta\in \mathbb{Z}}}\Big|\widehat{C}_{b^N}\Big(\theta+\frac{\ell}{d}+\frac{\eta}{b^N}\Big)\Big|\ll_b (b-s)^N(D^2B)^\alpha.
    \end{align*} Thus, we have that \begin{align}
        \Sigma_1&\ll_b N^4(b-s)^{N}(D^2B)^\alpha\Big(b^{\frac{4N}{5}}+\frac{b^N}{(DB)^{1/2}}+(DB)^{1/2}b^{N/2}\Big)\nonumber \\ \label{eq:minor1} &=N^4b^N(b-s)^N\Big(b^{-\frac{N}{5}}(D^2B)^\alpha+\frac{(D^2B)^\alpha}{(DB)^{1/2}}+\frac{(D^2B)^\alpha(DB)^{1/2}}{b^{N/2}}\Big).
    \end{align} Then, since $D^2B<b^N$, $DB<b^N/D_0$, and $B,D\gg 1$ by assumption, we have \begin{align*}
        b^{-\frac{N}{5}}(D^2B)^\alpha&<(D^2B)^{\alpha-\frac{1}{5}} \\
        \frac{(D^2B)^\alpha}{(DB)^{1/2}}&<(D^2B)^{\alpha - \frac{1}{5}} \\
        \frac{(D^2B)^\alpha(DB)^{1/2}}{b^{N/2}}&<\frac{b^{\alpha N}}{D_0^{1/2}}
    \end{align*} so that \begin{align*}
        \Sigma_1\ll_b N^4b^N(b-s)^N\Big((D^2B)^{\alpha-\frac{1}{5}}+\frac{b^{\alpha N}}{D_0^{1/2}}\Big).
    \end{align*} We now turn to the second sum $\Sigma_2$. By partial summation, we observe that $\widehat{\Lambda}_{b^N}\Big(\theta+O(b^{-N})\Big)$ obeys the same bound as $\widehat{\Lambda}_{b^N}(\theta)$ in Lemma \ref{lem:vonMangoldtFourierBounds}, and so we may deduce that \begin{align}\label{eq:pertubatedvonMangoldt}
        \sup_{\substack{d\sim D\\ (\ell,d)=1\\ |\eta|\ll 1}}\Big|\sum_{n<b^N}\Lambda(n)e\Big(-n\Big(\frac{\ell}{d}+\frac{\eta}{b^N}\Big)\Big)\Big|\ll_b N^4\Big(b^{4N/5}+\frac{b^N}{D^{1/2}}+\frac{D^{1/2}}{b^{N/2}}\Big).
    \end{align} This provides then, analogous to (\ref{eq:minor1}) with $B=1$, that \begin{align*}
        \Sigma_2\ll_b N^4b^N(b-s)^N\Big(b^{-\frac{N}{5}}D^{2\alpha}+D^{2\alpha-\frac{1}{2}}+\frac{D^{2\alpha+\frac{1}{2}}}{b^{N/2}}\Big).
    \end{align*} Then, since $1\ll D\ll D_0\ll b^{N/2}$ and $0<\alpha<1/5$ (from the standing conditions in \S\ref{subsec:standingAssumptions}) we have \begin{align*}
        b^{-\frac{N}{5}}D^{2\alpha}&<D^{-2(\frac{1}{5}-\alpha)}<D^{-(\frac{1}{5}-\alpha)} \\
        D^{2\alpha-\frac{1}{2}}&<D^{-(\frac{1}{5}-\alpha)}
    \end{align*} to provide the result.
\end{proof}

\subsection{An Inversion Theorem}

The key point of these estimates is that one may essentially control the Fourier transform of the restricted digit set with a small number of frequencies.

\begin{proposition}[Inversion with Few Frequencies]\label{prop:inversion}
    Take $\theta\in \mathbb{T}$ and $x\in \mathbb{T}$. Suppose the base $b$ is at least 5. Then, for $A>0$ and sufficiently large $B$ in terms of $A$, \begin{align*}
        \sum_{\substack{|\eta|<\log^B(b^N)\\ b^Nx+\eta\in \mathbb{Z}}}\widehat{C}_{b^N}\Big(\theta+x+\frac{\eta}{b^N}\Big)\sum_{k=0}^{b^N-1}e\Big(-\frac{k\eta}{b^N}\Big)=b^N\widehat{C}_{b^N}(\theta+x)+O\Big(\frac{b^N(b-s)^N}{\log^A(b^N)}\Big)
    \end{align*}
\end{proposition}

To prove the proposition, we need a supplemental lemma.

\begin{lemma}\label{lem:largeFourierPoints}
    Fix $b\geq 5$. There exists a constant $c_b>0$ depending only on $b$ such that the following holds. Let $I=\{h,h+1,...,h+|I|-1\}\subset \mathbb{Z}$ be an interval of cardinality $1\leq |I|\leq b^N$. Then, for $\lambda\geq 1$, $\theta\in \mathbb{T}$, \begin{align*}
        \sum_{k\in I}\mathbf{1}\Big(\Big|\widehat{C}_{b^N}\Big(\theta+\frac{k}{b^N}\Big)\Big|\geq \frac{(b-s)^N}{\lambda}\Big)\ll_b {|I|}^{\frac{2\log 2}{\log b}}\lambda^{c_b}.
    \end{align*} We may take $c_b=4b^3\log(\frac{b-2}{2})$.
\end{lemma}

\begin{proof}
    Since \begin{align*}
        \mathbf{1}\Big(\Big|\widehat{C}_{b^N}\Big(\theta+\frac{k}{b^N}\Big)\Big|>\frac{(b-s)^N}{\lambda}\Big)&\leq \mathbf{1}\Big(\sum_{i=0}^{N-1}\Big\|b^i\Big(\theta+\frac{k}{b^N}\Big)\Big\|^2<b\log \lambda\Big)
    \end{align*} we may bound the sum above by \begin{align*}
        \sum_{\substack{k\in I}}\mathbf{1}\Big(\sum_{i=0}^{N-1}\Big\|b^i\Big(\theta+\frac{k}{b^N}\Big)\Big\|^2<b\log \lambda\Big).
    \end{align*} Set $T:=\{k\in I:\sum_{i=0}^{N-1}\|b^i(\theta+\frac{k}{b^N})\|^2<b\log \lambda\}$. Suppose $k_1,k_2\in T$, then if $j:=k_2-k_1$ we have that $|k_2-k_1|\leq |I|$, and \begin{align*}
        \sum_{i=0}^{N-1}\|b^i(j/b^N)\|^2< 4b\log \lambda
    \end{align*} by Minkowski's inequality. Set $T_0:=\{j\in \mathbb{Z}:|j|\leq |I|,\sum_{i=0}^{N-1}\|b^i(j/b^N)\|^2<4b\log \lambda\}$, then $k_2-k_1\in T_0$. Since $k_1,k_2$ were arbitrary elements of $T$, we then have that $T-T\subset T_0$, and so $|T|\leq |T_0|$.

    We now show that $|T_0|\ll_b |I|^{\frac{2\log 2}{\log b}}\lambda^{c_b}$. Take $j\in T_0$, then we may write $j=\pm \sum_{c=0}^{m}a_cb^c$ with $a_c\in \{0,...,b-1\}$ and $m=\lfloor \frac{\log |I|}{\log b}\rfloor$. Set $a_c=0$ for $m<c<N$. Consider, for some $0\leq i<N$, the quantity $\|b^{i-N}j\|$. We may observe \begin{align*}
        \|b^{i-N}j\|=\Big\|b^{i-N}\sum_{c=0}^{m}a_cb^c\Big\|=\Big\|b^{i-N}\sum_{c=0}^{N-i-1}a_cb^c\Big\|&\geq \|a_{N-i-1}/b\|-\Big\|b^{i-N}\sum_{c=0}^{N-i-2}a_cb^c\Big\| \\ &\geq \|a_{N-i-1}/b\|-1/b
    \end{align*} and so if $a_{N-i-1}\not\in \{0,1,b-1\}$ this is at least $1/b$. Moreover, if $a_{N-i-1}=1$ then we have \begin{align*}
        \Big\|b^{i-N}\sum_{c=0}^{N-i-1}a_cb^c\Big\|=b^{i-N}\sum_{c=0}^{N-i-1}a_cb^c\geq \frac{1}{b}\quad (b\geq 4)
    \end{align*} and so $\|b^{i-N}j\|\geq 1/b$ if $a_{N-i-1}\not\in \{0,b-1\}$. Thus, \begin{align*}
        \sum_{i=0}^{N-1}\|b^i(j/b^N)\|^2\geq b^{-2}\#\{0\leq i<N: a_{i}\not\in \{0,b-1\}\}.
    \end{align*} Since $j\in T_0$ by assumption, we then have that \begin{align*}
        \#\{0\leq i<N:a_i\not\in \{0,b-1\}\}< 4b^3\log \lambda
    \end{align*} and in particular, \begin{align*}
        \#\{0\leq c\leq m:a_c\not\in \{0,b-1\}\}<4b^3\log \lambda.
    \end{align*} The problem of estimating $|T_0|$ is then reduced to that of counting tuples $(a_0,...,a_m)$ with $a_c\in \{0,...,b-1\}$ and $\#\{0\leq c\leq m:a_c\not\in \{0,b-1\}\}\leq 4b^3\log \lambda$. This quantity is bounded above by \begin{align*}
        \sum_{k=0}^{\lfloor 4b^3\log \lambda\rfloor}\binom{m+1}{k}(b-2)^k2^{m+1-k},
    \end{align*} and since $(b-2)^k2^{m+1-k}=2^{m+1}\Big(\frac{b-2}{2}\Big)^{k}\leq 2^{m+1}\Big(\frac{b-2}{2}\Big)^{4b^3\log \lambda}$ and $\sum_{k=0}^{\lfloor 4b^3\log \lambda\rfloor}\binom{m+1}{k}\leq 2^{m+1}$, we may bound \begin{align*}
        |T_0|\ll 4^m \lambda^{4b^3\log(\frac{b-2}{2})}.
    \end{align*} Finally, since $m\leq \frac{\log|I|}{\log b}$ we have that \begin{align*}
        |T_0|\ll_b {|I|}^{\frac{2\log 2}{\log b}}\lambda^{4b^3\log(\frac{b-2}{2})}
    \end{align*} to complete the proof with $c_b=4b^3\log(\frac{b-2}{2})$.
\end{proof}

\begin{proof}[Proof of Proposition \ref{prop:inversion}]

Let $J\subset \mathbb{Z}-b^Nx$ be an interval of cardinality $b^N$ containing $[-\log^B(b^N),\log^B(b^N)]$ (for concreteness, take $J=[-b^N/2,b^N/2)\cap (\mathbb{Z}-b^Nx)$) and consider first the completed sum \begin{align*}
    \sum_{\eta\in J}\widehat{C}_{b^N}\Big(\theta+x+\frac{\eta}{b^N}\Big)\sum_{k=0}^{b^N-1}e\Big(-\frac{k\eta}{b^N}\Big).
\end{align*} By expanding out the Fourier transform and interchanging summations, this is precisely \begin{align*}
    \sum_{n<b^N}\mathbf{1}_\mathcal{C}(n)e(n(\theta+x))\sum_{k=0}^{b^N-1}\sum_{\eta\in J}e\Big(\frac{\eta (n-k)}{b^N}\Big)=b^N\widehat{C}_{b^N}(\theta+x).
\end{align*} Thus, we have that \begin{align*}
    b^N\widehat{C}_{b^N}(\theta+x)-\sum_{\substack{|\eta|<\log^B(b^N)\\ b^Nx+\eta\in \mathbb{Z}}}\widehat{C}_{b^N}\Big(\theta+x+\frac{\eta}{b^N}\Big)\sum_{k=0}^{b^N-1}e\Big(-\frac{k\eta}{b^N}\Big)\\ =\sum_{\substack{\eta\in J\\ |\eta|\geq \log^B(b^N)}}\widehat{C}_{b^N}\Big(\theta+x+\frac{\eta}{b^N}\Big)\sum_{k=0}^{b^N-1}e\Big(-\frac{k\eta}{b^N}\Big)
\end{align*} and so \begin{align*}
    E:=\Big|b^N\widehat{C}_{b^N}(\theta+x)-\sum_{\substack{|\eta|<\log^B(b^N)\\ b^Nx+\eta\in \mathbb{Z}}}\widehat{C}_{b^N}\Big(\theta+x+\frac{\eta}{b^N}\Big)\sum_{k=0}^{b^N-1}e\Big(-\frac{k\eta}{b^N}\Big)\Big|\\ \leq \sum_{\substack{\eta\in J\\ |\eta|\geq \log^B(b^N)}}\Big|\widehat{C}_{b^N}\Big(\theta+x+\frac{\eta}{b^N}\Big)\Big|\cdot \Big|\sum_{k=0}^{b^N-1}e\Big(-\frac{k\eta}{b^N}\Big)\Big|.
\end{align*} Using that $|\sum_{k=0}^{b^N-1}e(-\frac{k\eta}{b^N})|\ll \|\eta/b^N\|^{-1}$ we then have that this error $E$ satisfies \begin{align*}
    E\ll \sum_{\substack{\eta\in J\\ |\eta|\geq \log^B(b^N)}}\Big|\widehat{C}_{b^N}\Big(\theta+x+\frac{\eta}{b^N}\Big)\Big|\cdot \|\eta/b^N\|^{-1}.
\end{align*} For a parameter $\lambda\geq 1$ to be determined later, we will partition the points $\{\eta\in J:|\eta|\geq \log^B(b^N)\}$ into two categories: where $|\widehat{C}_{b^N}(\theta+x+\frac{\eta}{b^N})|\geq \frac{(b-s)^N}{\lambda}$, and where $|\widehat{C}_{b^N}(\theta+x+\frac{\eta}{b^N})|<\frac{(b-s)^N}{\lambda}$. This gives that \begin{align*}
    E\ll \Sigma_1+\Sigma_2,
\end{align*} where \begin{align*}
    \Sigma_1&:=(b-s)^N\sum_{\substack{\eta\in J\\ |\eta|\geq \log^B(b^N)}}\mathbf{1}\Big(\Big|\widehat{C}_{b^N}\Big(\theta+x+\frac{\eta}{b^N}\Big)\Big|\geq \frac{(b-s)^N}{\lambda}\Big)\cdot \|\eta/b^N\|^{-1} \\ \Sigma_2&:=\frac{(b-s)^N}{\lambda}\sum_{\substack{\eta\in J\\ |\eta|\geq \log^B(b^N)}}\|\eta/b^N\|^{-1}.
\end{align*} It is easy to observe that $\sum_{\substack{\eta\in J\\ |\eta|\geq \log^B(b^N)}}\|\eta/b^N\|^{-1}\ll b^N\log(b^N)$, and so \begin{align*}\Sigma_2\ll \frac{b^N(b-s)^N\log(b^N)}{\lambda}.\end{align*}

To bound $\Sigma_1$, we will use partial summation. First consider where $\log^B(b^N)\leq \eta\leq b^N/2$, and set $f(\eta):=\mathbf{1}(\eta\geq \log^B(b^N))\cdot \mathbf{1}\Big(\Big|\widehat{C}_{b^N}\Big(\theta+x+\frac{\eta}{b^N}\Big)\Big|\geq \frac{(b-s)^N}{\lambda}\Big)$. Here, $\|\eta/b^N\|=\eta/b^N$, and so \begin{align*}
    \sum_{\substack{\eta\in J\\ \log^B(b^N)\leq \eta\leq b^N/2}}f(\eta) \|\eta/b^N\|^{-1} =b^N\sum_{\substack{\eta\in J\\ \log^B(b^N)\leq \eta\leq b^N/2}}f(\eta) \eta^{-1}.
\end{align*} By partial summation, we may bound this above by \begin{align*}
    b^N\Big(b^{-N}\sum_{\substack{\eta\in J\\ \log^B(b^N)\leq \eta\leq b^N/2}}f(\eta)+\int_{\log^B(b^N)}^{b^N/2}\frac{1}{t^2}\sum_{\substack{\eta\in J\\ \log^B(b^N)\leq \eta\leq t}}f(\eta) dt\Big).
\end{align*} Applying Lemma \ref{lem:largeFourierPoints} gives that the first sum is bounded above by $\lambda^{c_b}b^{\frac{2\log 2}{\log b}N}$, and that the sum inside the integral is bounded above by $\lambda^{c_b}t^{\frac{2\log 2}{\log b}}$, and so the expression is bounded above by \begin{align*}
    \lambda^{c_b}b^{\frac{2\log 2}{\log b}N}+\lambda^{c_b}b^N\int_{\log^B(b^N)}^{b^N/2}t^{-2+\frac{2\log 2}{\log b}}dt\ll \lambda^{c_b}b^N\log^{-(1-\frac{2\log 2}{\log b})B}.
\end{align*} The case where $-b^N/2\leq \eta\leq -\log^B(b^N)$ follows similarly, and so we may then deduce that \begin{align*}
    \Sigma_1\ll b^N(b-s)^N\lambda^{c_b}\log^{-(1-\frac{2\log 2}{\log b})B}.
\end{align*}

Thus, \begin{align*}
    \Sigma_1+\Sigma_2\ll_b b^N(b-s)^N\Big(\lambda^{c_b}\log^{-(1-\frac{2\log 2}{\log b})B}(b^N)+\frac{\log(b^N)}{\lambda}\Big).
\end{align*} 

Choosing $\lambda=\log^{A+1}(b^N)$, we see that for sufficiently large $B$ the result holds.

\end{proof}

With Proposition \ref{prop:inversion}, we can now simplify the exponential sums that will arise from Dirichlet's approximation theorem later.

\begin{lemma}\label{lem:arcReduction}
    Take $\theta\in \mathbb{T}$. Then, for $A>0$ and sufficiently large $B$ in terms of $A$, \begin{align*}
        \sum_{d<\log^A(b^N)}\sum_{\ell\in (\mathbb{Z}/d\mathbb{Z})^*}\sum_{\substack{|\eta|<\log^B(b^N)\\ b^N\ell/d+\eta\in \mathbb{Z}}}\widehat{C}_{b^N}\Big(\theta+\frac{\ell}{d}+\frac{\eta}{b^N}\Big)\widehat{\Lambda}_{b^N}\Big(-\frac{\ell}{d}-\frac{\eta}{b^N}\Big) \\ =b^N\sum_{d<\log^A(b^N)}\frac{\mu(d)}{\phi(d)}\sum_{\ell\in (\mathbb{Z}/d\mathbb{Z})^*}\widehat{C}_{b^N}\Big(\theta+\frac{\ell}{d}\Big)+O_A\Big(\frac{b^N(b-s)^N}{\log^A(b^N)}\Big).
    \end{align*}
\end{lemma}

\begin{proof}
    We use Proposition \ref{prop:inversion} alongside the estimate \begin{align}
        \widehat{\Lambda}_{b^N}\Big(-\frac{\ell}{d}-\frac{\eta}{b^N}\Big)=\frac{\mu(d)}{\phi(d)}\sum_{k=0}^{b^N-1}e\Big(-\frac{\eta k}{b^N}\Big)+O_C\Big(\frac{b^N}{\log^C(b^N)}\Big),
    \end{align} which follows from the Siegel-Walfisz theorem and partial summation.
\end{proof}

We now define our major arcs. For $A',B>0$, define \begin{align*}
    \mathfrak{M}:=\Big\{(d,\ell,\eta):1\leq d\leq \log^{A'}(b^N),(\ell,d)=1,|\eta|\leq \log^B(b^N), b^N\ell/d+\eta\in \mathbb{Z}\Big\}.
\end{align*} Each triple $(d,\ell,\eta)$ represents the Fourier gridpoint $\frac{a}{b^N}=\frac{\ell}{d}+\frac{\eta}{b^N}\ (\text{mod }1)$.\\

\begin{lemma}[Uniqueness of a major representation] For sufficiently large $N$ in terms of $A'$ and $B$, a Fourier gridpoint $\frac{a}{b^N}$ has at most one representation by a triple in $\mathfrak{M}$.
\end{lemma}

\begin{proof}
    Suppose $\frac{a}{b^N}=\frac{\ell}{d}+\frac{\eta}{b^N}=\frac{\ell'}{d'}+\frac{\eta'}{b^N}$, where $(d,\ell,\eta),(d',\ell',\eta')\in \mathfrak{M}$. First consider the case where $(d,\ell)=(d',\ell')$. Then, $\frac{\eta}{b^N}= \frac{\eta'}{b^N}\ (\text{mod }1)$, and so $\eta\in \eta'+b^N\mathbb{Z}$. Since $|\eta|\leq \log^B(b^N)$, this provides for sufficiently large $N$ in terms of $B$ that $\eta=\eta'$. Now consider where $(d,\ell)\neq (d',\ell')$. Then, since $d,d'\leq \log^{A'}(b^N)$, one observes that $|\frac{\ell}{d}-\frac{\ell'}{d'}|\geq \log^{-2A'}(b^N)$, and so \begin{align*}
        \log^{-2A'}(b^N)\leq \frac{|\eta-\eta'|}{b^N}\leq \frac{2\log^B(b^N)}{b^N}.
    \end{align*} This is a contradiction for sufficiently large $N$ in terms of $B$ and $A'$.
\end{proof}

\begin{proposition}[Low-dimensional approximant]\label{prop:ExponentialSumReduction} Take $\theta\in \mathbb{T}$. Then, for any $A>0$, one has that \begin{align*}
    \sum_{n<b^N}\mathbf{1}_\mathcal{C}(n)\Lambda(n)e(n\theta)=\sum_{d<\log^{A'}(b^N)}\frac{\mu(d)}{\phi(d)}\sum_{\ell \in (\mathbb{Z}/d\mathbb{Z})^*}\widehat{C}_{b^N}\Big(\theta+\frac{\ell}{d}\Big)+O_A\Big(\frac{(b-s)^N}{\log^A(b^N)}\Big)
\end{align*} for sufficiently large $A'>0$ (that depends on $A$).\end{proposition}

\begin{proof}
   By Fourier inversion we may write \begin{align*}
       \sum_{n<b^N}\mathbf{1}_\mathcal{C}(n)\Lambda(n)e(n\theta)=b^{-N}\sum_{0\leq a<b^N}\widehat{C}_{b^N}\Big(\theta+\frac{a}{b^N}\Big)\widehat{\Lambda}_{b^N}\Big(-\frac{a}{b^N}\Big).
   \end{align*} We separate out Fourier gridpoints $\frac{a}{b^N}$ that can be represented by a triple in $\mathfrak{M}$. For every remaining grid point $\frac{a}{b^N}$, and setting $D_0:=\lfloor b^{N/2}/10\rfloor$, Dirichlet's theorem gives a reduced fraction $\frac{\ell}{d}$ with $1\leq d\leq D_0$, such that $\|\frac{a}{b^N}-\frac{\ell}{d}\|\leq \frac{1}{dD_0}$. Choose one such approximant by a fixed deterministic rule, e.g. by first minimizing $d$ and then choosing the least admissible numerator. Define \begin{align*}
       \eta:=a-\frac{b^N\ell}{d},\quad |\eta|\leq \frac{b^N}{dD_0}.
   \end{align*} Because $\frac{a}{b^N}$ is not represented by a triple in $\mathfrak{M}$, we must have that either $d>\log^{A'}(b^N)$ or $|\eta|>\log^B(b^N)$.
   
   This gives a partition of the Fourier gridpoints $\frac{a}{b^N}$, and so we may write \begin{align*}
       \sum_{n<b^N}\mathbf{1}_\mathcal{C}(n)\Lambda(n)e(n\theta)=b^{-N}\sum_{(d,\ell,\eta)\in \mathfrak{M}}\widehat{C}_{b^N}\Big(\theta+\frac{\ell}{d}+\frac{\eta}{b^N}\Big)\widehat{\Lambda}_{b^N}\Big(-\frac{\ell}{d}-\frac{\eta}{b^N}\Big)+E,
   \end{align*} where $E$ is the sum over the minor triples (i.e. nonmajor triples). To bound $E$ in absolute value, one may first apply the triangle inequality to induce positive summands, and then majorize the sum over minor triples by the sum over the larger set of triples \begin{align*}
       E_1\cup E_2&:=\Big\{(d,\ell,\eta):\log^{A'}(b^N)<d\leq D_0,(\ell,d)=1,|\eta|\leq \frac{b^N}{dD_0},b^N\ell/d+\eta\in \mathbb{Z}\Big\}\\ &\cup \Big\{(d,\ell,\eta):d\leq \log^{A'}(b^N),(\ell,d)=1,\log^B(b^N)<|\eta|\leq \frac{b^N}{dD_0},b^N\ell/d+\eta\in \mathbb{Z}\Big\},
   \end{align*} so that \begin{align*}
       |E|&\leq b^{-N}\sum_{(d,\ell,\eta)\in E_1}\Big|\widehat{C}_{b^N}\Big(\theta+\frac{\ell}{d}+\frac{\eta}{b^N}\Big)\widehat{\Lambda}_{b^N}\Big(-\frac{\ell}{d}-\frac{\eta}{b^N}\Big)\Big|\\ &+b^{-N}\sum_{(d,\ell,\eta)\in E_2}\Big|\widehat{C}_{b^N}\Big(\theta+\frac{\ell}{d}+\frac{\eta}{b^N}\Big)\widehat{\Lambda}_{b^N}\Big(-\frac{\ell}{d}-\frac{\eta}{b^N}\Big)\Big|.
   \end{align*} Let $A'>0$ be determined later. Choose $B$ sufficiently large in terms of $A$ so that we may apply Proposition \ref{prop:inversion}. For the purposes of Lemma \ref{lem:minorArcs}, we view $B$ as comparable to $|\eta|$, and $D$ as comparable to $d$. We first consider the sum over $E_1$. By partitioning the range of $d$ and $\eta$ into dyadic intervals, applying the first part of Lemma \ref{lem:minorArcs} to the regime where $|\eta|\gg 1$, and applying the second part to where $|\eta|\ll 1$, we obtain that the total contribution of this first sum is \begin{align*}
       \ll_{A'} N^6b^N(b-s)^N\Big(\frac{1}{\log^{(\frac{1}{5}-\alpha)A'}(b^N)}+\frac{b^{\alpha N}}{D_0^{1/2}}\Big)+N^6b^N(b-s)^N\frac{D_0^{2\alpha+\frac{1}{2}}}{b^{N/2}}
   \end{align*} (here we use that $\alpha<\frac{1}{5}$, to ensure that the dyadic sums are bounded). Similarly, applying the first part of Lemma \ref{lem:minorArcs} to a dyadic decomposition of the sum over $E_2$, the contribution here is \begin{align*}
       \ll_{B}N^6b^N(b-s)^N\Big(\frac{1}{\log^{(\frac{1}{5}-\alpha)B}(b^N)}+\frac{b^{\alpha N}}{D_0^{1/2}}\Big).
   \end{align*} We then obtain with $\alpha<\frac{1}{5}$ and $D_0\ll b^{N/2}$ that $E$ may be bounded as \begin{align*}
       |E|\ll_{A',B} (b-s)^N\cdot (N^{6-(\frac{1}{5}-\alpha)A'}+N^{6-(\frac{1}{5}-\alpha)B}).
   \end{align*} Choosing $A'$ and $B$ sufficiently large in terms of $A$ to ensure that $|E|\ll_A (b-s)^N \log^{-A}(b^N)$, we are left with \begin{align*}
       \sum_{n<b^N}\mathbf{1}_\mathcal{C}(n)\Lambda(n)e(n\theta)=b^{-N}\sum_{(d,\ell,\eta)\in \mathfrak{M}}\widehat{C}_{b^N}\Big(\theta+\frac{\ell}{d}+\frac{\eta}{b^N}\Big)\widehat{\Lambda}_{b^N}\Big(-\frac{\ell}{d}-\frac{\eta}{b^N}\Big).
   \end{align*} We may then apply Lemma \ref{lem:arcReduction} to simplify the main term here, and so \begin{align*}
       \sum_{n<b^N}\mathbf{1}_\mathcal{C}(n)\Lambda(n)e(n\theta)=\sum_{d\leq \log^{A'}(b^N)}\frac{\mu(d)}{\phi(d)}\sum_{\ell\in (\mathbb{Z}/d\mathbb{Z})^*}\widehat{C}_{b^N}\Big(\theta+\frac{\ell}{d}\Big)+O_A\Big(\frac{(b-s)^N}{\log^A(b^N)}\Big).
   \end{align*}

\end{proof}

\section{Proof of the Main Theorems}

\subsection{Proof of Theorem \ref{thm:Dirichlet}}

In this section, we prove the following main result, which is Theorem \ref{thm:Dirichlet} restated.

\begin{theorem}\label{thm:Dirichlet(2)}
    Let $m\geq 1$ and $t\in \mathbb{Z}/m\mathbb{Z}$. Then, for any $K>0$, \begin{align*}
        \sum_{\substack{n<b^N\\ n\equiv t\ (\text{mod }m)}}\mathbf{1}_\mathcal{C}(n)\Lambda(n)=\kappa_{m,t}(b-s)^N+O_K\Big(\frac{(b-s)^N}{\log^K(b^N)}\Big)
    \end{align*} where \begin{align*}
        \kappa_{m,t}:=\frac{b}{\phi(bv)}\mathbf{1}((t,v)=1)\cdot (b-s)^{-L}\sum_{\substack{n<b^L\\ n\equiv t\ (\text{mod }u)}}\mathbf{1}_\mathcal{C}(n)\mathbf{1}((n,b)=1).
    \end{align*} where we write $m=uv$, such that $(v,b)=1$ and $p|u\implies p|b$. $L$ is any sufficiently large number such that $u|b^{L}$ (the normalized sum is independent of sufficiently large $L$).

\end{theorem}

To prove the theorem, we will need an auxiliary lemma.

\begin{lemma} Fix $m\in \mathbb{N}$, and write $m=uv$ with $p|u\implies p|b$ and $(v,b)=1$. Take $h\in \mathbb{Z}$ such that $hb\equiv 1\ (\text{mod }v)$. Suppose we are given $d|bv$ and $\ell \in (\mathbb{Z}/d\mathbb{Z})^*$, and that $b^k(a/m+\ell/d)\in \mathbb{Z}$ for some $k\in \mathbb{N}$. Then, the following are true: \begin{enumerate}
    \item[(i)] $a\equiv -(\frac{bm}{d})\ell h\ (\text{mod }v)$
    \item[(ii)] There exists some $L=L(m)$ depending only on $m$ such that $b^{L}(a/m+\ell/d)\in \mathbb{Z}$. We may take $L$ to be any positive integer sufficiently large so that $u|b^{L}$.
\end{enumerate} 

\end{lemma}

\begin{proof}
    We first show (i). Suppose that $a_1,a_2\in \mathbb{Z}/m\mathbb{Z}$ satisfy $b^{k_i}(a_i/m+\ell/d)\in \mathbb{Z}$, for $i=1,2$. Then, $dm|b^{k_i}(a_id+\ell m)$ for $i=1,2$. Without loss of generality, we may take $k_2\geq k_1$. Then, $dm|b^{k_2}(a_id+\ell m)$ for $i=1,2$, and so $dm|b^{k_2}d(a_2-a_1)$. Since $v|m$ and $(v,b)=1$, we then deduce that $v|(a_2-a_1)$. Now, choosing $a_0:=-(\frac{bm}{d})\ell h$, we may compute \begin{align*}
        \frac{a_0}{m}+\frac{\ell}{d}=\frac{(1-bh)\ell}{d}.
    \end{align*} Writing $d=u_dv_d$ with $u_d|b$ and $v_d|v$, we then have that \begin{align*}
        b\Big(\frac{a_0}{m}+\frac{\ell}{d}\Big)=\frac{1-bh}{v}\cdot \frac{v}{v_d}\cdot \frac{b}{u_d}\cdot \ell\in \mathbb{Z}.
    \end{align*} This provides (i).

    Now, suppose that $a\equiv a_0\ (\text{mod }v)$, and write $a=a_0+cv$ for some $c\in \mathbb{Z}$. Choose $L\in \mathbb{N}$ sufficiently large so that $u|b^L$. Then, \begin{align*}
        b^L\Big(\frac{a_0}{m}+\frac{\ell}{d}+\frac{cv}{m}\Big)=b^L\Big(\frac{a_0}{m}+\frac{\ell}{d}+\frac{c}{u}\Big)\in \mathbb{Z}.
    \end{align*}
\end{proof}

\begin{proof}[Proof of Theorem \ref{thm:Dirichlet(2)}] By orthogonality, we may write \begin{align*}
    \mathbf{1}(n\equiv t\ (\text{mod }m))=\frac{1}{m}\sum_{c=1}^me\Big(\frac{(n-t)c}{m}\Big),
\end{align*} and so \begin{align*}
    \sum_{\substack{n<b^N\\ n\equiv t\ (\text{mod }m)}}\mathbf{1}_\mathcal{C}(n)\Lambda(n)=\frac{1}{m}\sum_{c=1}^me(-ct/m)\sum_{n<b^N}\mathbf{1}_\mathcal{C}(n)\Lambda(n)e(c n/m).
\end{align*} Take $K>0$. We may then apply Proposition \ref{prop:ExponentialSumReduction} to deduce that there is some $A'>0$ such that this is precisely \begin{align*}
    \frac{1}{m}\sum_{c=1}^me(-ct/m)\sum_{d<\log^{A'}(b^N)}\frac{\mu(d)}{\phi(d)}\sum_{\ell\in (\mathbb{Z}/d\mathbb{Z})^*}\widehat{C}_{b^N}\Big(\frac{c}{m}+\frac{\ell}{d}\Big)+O_{K}\Big(\frac{(b-s)^N}{\log^K(b^N)}\Big).
\end{align*} Moving the sum over $c$ to the innermost position, the main term is \begin{align*}
    \sum_{d<\log^{A'}(b^N)}\frac{\mu(d)}{\phi(d)}\sum_{\ell\in (\mathbb{Z}/d\mathbb{Z})^*} \frac{1}{m}\sum_{c=1}^me(-ct/m)\widehat{C}_{b^N}\Big(\frac{c}{m}+\frac{\ell}{d}\Big).
\end{align*}

Now, for a given choice of $(d,\ell,c)$, if $b^N\Big(\frac{c}{m}+\frac{\ell}{d}\Big)\not\in \mathbb{Z}$, then $\|b^i(\frac{c}{m}+\frac{\ell}{d})\|\geq \frac{1}{md}$ for all $0\leq i<N$, and so by similar logic as Lemma \ref{lem:l-inftyBound} we may bound \begin{align*}
    \Big|\widehat{C}_{b^N}\Big(\frac{c}{m}+\frac{\ell}{d}\Big)\Big|\leq (b-s)^N\exp(-c_0N/\log N)
\end{align*} for a constant $c_0=c_0(A',b)$. Clearly, the contribution from such $(d,\ell,c)$ is negligible, and so we may restrict to $(d,\ell,c)$ that satisfy $b^N\Big(\frac{c}{m}+\frac{\ell}{d}\Big)\in \mathbb{Z}$. From the second part of the auxiliary lemma above, we observe that this implies $b^L\Big(\frac{c}{m}+\frac{\ell}{d}\Big)\in \mathbb{Z}$, where $L$ is as in the lemma, and depends only on $m$. Thus, for the non-negligible $(d,\ell,c)$, we have that \begin{align*}
    \widehat{C}_{b^N}\Big(\frac{c}{m}+\frac{\ell}{d}\Big)=\widehat{C}_{b^{N-L}}\Big(b^L\Big(\frac{c}{m}+\frac{\ell}{d}\Big)\Big)\widehat{C}_{b^L}\Big(\frac{c}{m}+\frac{\ell}{d}\Big)=(b-s)^{N-L}\widehat{C}_{b^L}\Big(\frac{c}{m}+\frac{\ell}{d}\Big).
\end{align*} This provides then that our main term is \begin{align*}
    (b-s)^{N-L}\sum_{d<\log^{A'}(b^N)}\frac{\mu(d)}{\phi(d)}\sum_{\ell \in (\mathbb{Z}/d\mathbb{Z})^*}\frac{1}{m}\sum_{1\leq c\leq m}'e(-ct/m)\widehat{C}_{b^L}\Big(\frac{c}{m}+\frac{\ell}{d}\Big),
\end{align*} where $\sum'$ denotes that we only sum over $c$ such that $b^N(\frac{c}{m}+\frac{\ell}{d})\in \mathbb{Z}$. Notice also that this implies that $d|b^Nm$, and since we may restrict to squarefree $d$, $d|bv$. Now, from the first part of the auxiliary lemma above, we must have such $c$ satisfying $c\equiv - (\frac{bm}{d})\ell h\ (\text{mod }v)$, and so we may write this as \begin{align*}
    (b-s)^{N-L}\sum_{d|bv}\frac{\mu(d)}{\phi(d)}\sum_{\ell \in (\mathbb{Z}/d\mathbb{Z})^*}\frac{1}{m}\sum_{c'=1}^u e\Big(-\frac{t}{m}(-\frac{bm\ell h}{d}+c'v)\Big)\widehat{C}_{b^L}\Big(\frac{(1-bh)\ell}{d}+\frac{c'}{u}\Big).
\end{align*} Expanding out the Fourier transform and rearranging terms, this is \begin{align*}
    (b-s)^{N-L}\sum_{d|bv}\frac{\mu(d)}{\phi(d)}\sum_{\ell \in (\mathbb{Z}/d\mathbb{Z})^*}\frac{e(bh\ell t/d)}{m}\sum_{n<b^L}\mathbf{1}_\mathcal{C}(n)e\Big(\frac{(1-bh)\ell n}{d}\Big)\sum_{c'=1}^ue\Big(\frac{(n-t)c'}{u}\Big).
\end{align*} By orthogonality, the innermost sum evaluates to $u\mathbf{1}(u|(n-t))$, and so this may be written as \begin{align*}
    (b-s)^{N-L}\frac{1}{v}\sum_{d|bv}\frac{\mu(d)}{\phi(d)}\sum_{\ell\in (\mathbb{Z}/d\mathbb{Z})^*}e(bh\ell t/d)\sum_{\substack{n<b^L\\ n\equiv t\ (\text{mod }u)}}\mathbf{1}_\mathcal{C}(n)e\Big(\frac{(1-bh)\ell n}{d}\Big).
\end{align*} Moving the sum over $\ell$ to the innermost position, we have Ramanujan's sum \begin{align*}
    \sum_{\ell \in (\mathbb{Z}/d\mathbb{Z})^*}e\Big(\frac{\ell}{d}(bht+(1-bh)n)\Big)=c_d(bht+(1-bh)n),
\end{align*} and so then moving the sum over $d$ to the innermost position, we have a main term of the form \begin{align*}
    \frac{(b-s)^{N-L}}{v}\sum_{\substack{n<b^L\\ n\equiv t\ (\text{mod }u)}}\mathbf{1}_\mathcal{C}(n)\sum_{d|bv}\frac{\mu(d)}{\phi(d)}c_d(bht+(1-bh)n).
\end{align*} Applying the so-called Brauer-Rademacher identity $\sum_{d|j}\frac{\mu(d)}{\phi(d)}c_d(k)=\frac{j}{\phi(j)}\mathbf{1}((j,k)=1)$, we obtain that this is precisely \begin{align}\label{eq:preliminaryConstant}
    \frac{b(b-s)^{N-L}}{\phi(bv)}\sum_{\substack{n<b^L\\ n\equiv t\ (\text{mod }u)}}\mathbf{1}_\mathcal{C}(n)\mathbf{1}((bht+(1-bh)n,bv)=1).
\end{align}

We now simplify the constant in (\ref{eq:preliminaryConstant}). Since $(b,v)=1$, we may factor \begin{align*}
    \mathbf{1}((k,bv)=1)&=\mathbf{1}((k,b)=1)\mathbf{1}((k,v)=1),\quad k:=bht+(1-bh)n.
\end{align*} Since $k\equiv n\ (\text{mod }b)$, the first factor is equivalent to $\mathbf{1}((n,b)=1)$. Whereas since $bh\equiv 1\ (\text{mod }v)$, we may also deduce that $k\equiv t\ (\text{mod }v)$, and so the second factor is equivalent to $\mathbf{1}((t,v)=1)$. Consequently, one has that \begin{align*}
    &\sum_{\substack{n<b^L\\ n\equiv t\ (\text{mod }u)}}\mathbf{1}_\mathcal{C}(n)\mathbf{1}((bht+(1-bh)n,bv)=1)\\&\quad\quad =\mathbf{1}((t,v)=1)\sum_{\substack{n<b^L\\ n\equiv t\ (\text{mod }u)}}\mathbf{1}_\mathcal{C}(n)\mathbf{1}((n,b)=1),
\end{align*} and so we are left with a main term of the form \begin{align*}
    \kappa_{m,t}(b-s)^N,
\end{align*} where \begin{align*}
    \kappa_{m,t}:=\frac{b}{\phi(bv)}\mathbf{1}((t,v)=1)\cdot (b-s)^{-L}\sum_{\substack{n<b^L\\ n\equiv t\ (\text{mod }u)}}\mathbf{1}_\mathcal{C}(n)\mathbf{1}((n,b)=1).
\end{align*} 
    
\end{proof}

\begin{lemma}[Well-definedness of the constant]\label{lem:wellDefined} Provided $u|b^L$, the quantity \begin{align*}
    (b-s)^{-L}\sum_{\substack{n<b^L\\ n\equiv t\ (\text{mod }u)}}\mathbf{1}_\mathcal{C}(n)\mathbf{1}((n,b)=1)
\end{align*} does not depend on $L$.
\end{lemma}

\begin{proof}
    Let $L_0$ be the minimal choice of $L\in \mathbb{N}$ such that $u|b^L$. Take any $L'\geq L_0$. Then, by partitioning $\mathbb{Z}_{\geq 0}$ into residue classes modulo $b^{L_0}$, one has that \begin{align*}
        &\sum_{\substack{n<b^{L'}\\ n\equiv t\ (\text{mod }u)}}\mathbf{1}_\mathcal{C}(n)\mathbf{1}((n,b)=1) \\ &=\sum_{c=0}^{b^{L_0}-1}\sum_{m<b^{L'-L_0}}\mathbf{1}_\mathcal{C}(b^L_0m+c)\mathbf{1}\Big((b^{L_0}m+c,b)=1\Big)\mathbf{1}\Big(b^{L_0}m+c\equiv t\ (\text{mod } u)\Big) \\ &=\sum_{c=0}^{b^{L_0}-1}\sum_{m<b^{L'-L_0}}\mathbf{1}_\mathcal{C}(m)\mathbf{1}_\mathcal{C}(c)\mathbf{1}((c,b)=1)\mathbf{1}\Big(c\equiv t\ (\text{mod }u)\Big) \\ &= (b-s)^{L'-L_0}\sum_{\substack{c<b^{L_0}\\ c\equiv t\ (\text{mod }u)}}\mathbf{1}_\mathcal{C}(c)\mathbf{1}((c,b)=1).
    \end{align*} By dividing both sides of the equation by $(b-s)^{L'}$, we observe the result.
\end{proof}

\subsection{Proof of Theorem \ref{thm:Vinogradov}}

In this section, we prove an analogue of Vinogradov's theorem for exponential sums over primes. This is Theorem \ref{thm:Vinogradov}, restated.

\begin{theorem}\label{thm:VinogradovII}
    Suppose $\mathcal{C}$ satisfies the standing conditions in \S\ref{subsec:standingAssumptions}. Then, for any $\theta\in \mathbb{R}\setminus \mathbb{Q}$, \begin{align*}
        \sum_{n<b^N}\mathbf{1}_\mathcal{C}(n)\Lambda(n)e(n\theta)=o((b-s)^N).
    \end{align*}
\end{theorem}

We will need another auxiliary lemma for the proof of this theorem.

\begin{lemma}[At most one exceptional denominator]\label{lem:uncertainty} Fix $\theta\in \mathbb{T}$ and $u|b$. Take $A'>0$ and $N$ sufficiently large in terms of $A'$. Then, among denominators $v<\log^{A'}(b^N)$ with $(v,b)=1$, there is at most one $v$ such that \begin{align*}
    \max_{\ell \in (\mathbb{Z}/uv\mathbb{Z})^\times}\Big|\widehat{C}_{b^N}\Big(\theta+\frac{\ell}{uv}\Big)\Big|>(b-s)^N\exp\Big(-\frac{1}{b}N^{2/3}\Big).
\end{align*} Moreover, there are at most $u$ choices of $\ell \in (\mathbb{Z}/uv\mathbb{Z})^\times$ where this inequality is attained.
\end{lemma}

\begin{proof}
    Suppose that we had some $v_1, v_2$ with $(v_i,b)=1$, and $\ell_1\in (\mathbb{Z}/uv_1\mathbb{Z})^\times$, $\ell_2\in (\mathbb{Z}/uv_2\mathbb{Z})^\times$ such that \begin{align*}
        \Big|\widehat{C}_{b^N}\Big(\theta+\frac{\ell_1}{uv_1}\Big)\Big|&>(b-s)^N\exp(-\frac{1}{b}N^{2/3}),\\ \Big|\widehat{C}_{b^N}\Big(\theta+\frac{\ell_2}{uv_2}\Big)\Big|&>(b-s)^N\exp(-\frac{1}{b}N^{2/3}).
    \end{align*} Notice, by the triangle inequality and Cauchy-Schwarz, \begin{align*}
        &\sum_{i=0}^{N-1}\Big\|b^i\Big(\frac{\ell_1}{uv_1}-\frac{\ell_2}{uv_2}\Big)\Big\|\\ &\leq N^{1/2}\Big(\sum_{i=0}^{N-1}\Big\|b^i\Big(\theta+\frac{\ell_1}{uv_1}\Big)\Big\|^2\Big)^{1/2}+N^{1/2}\Big(\sum_{i=0}^{N-1}\Big\|b^i\Big(\theta+\frac{\ell_2}{uv_2}\Big)\Big\|^2\Big)^{1/2}.
    \end{align*} By the assumption and the inequality \begin{align*}
        |\widehat{C}_{b^N}(t)|\leq (b-s)^N\exp\Big(-\frac{1}{b}\sum_{i=0}^{N-1}\|b^it\|^2\Big)
    \end{align*} (from the derivation of the $L^\infty$ bound) we obtain that \begin{align*}
        \sum_{i=0}^{N-1}\Big\|b^i\Big(\theta+\frac{\ell_1}{uv_1}\Big)\Big\|^2\leq N^{2/3},\quad \sum_{i=0}^{N-1}\Big\|b^i\Big(\theta+\frac{\ell_2}{uv_2}\Big)\Big\|^2\leq N^{2/3}
    \end{align*} and so we then have that \begin{align*}
        \sum_{i=0}^{N-1}\Big\|b^i\Big(\frac{\ell_1}{uv_1}-\frac{\ell_2}{uv_2}\Big)\Big\|\ll N^{5/6}.
    \end{align*} Write $\frac{\ell_1}{uv_1}-\frac{\ell_2}{uv_2}=\frac{k}{j}$ with $(k,j)=1$. We have two cases from here: either $p|j\implies p|b$, or there exists some prime $p|j$ such that $p\nmid b$. Consider first the second case. Then, $\|b^ik/j\|\geq \frac{1}{j}\geq \frac{1}{(\log^{A'}(b^N))^2}$ for all $0\leq i<N$, and so for any interval of size $\frac{2\log (\log^{A'}(b^N))}{\log b}$ in $\{0,...,N-1\}$ we may find some index $i$ such that $\|b^ik/j\|\geq \frac{1}{2b}$. This gives that the sum above is $\gg_b \frac{N}{\log (\log^{A'}(b^N))}$, which is a contradiction for sufficiently large $N$ in terms of $A'$.

    We are left with the case where $p|j\implies p|b$, which implies that there exists some $i\geq 0$ such that $u^2v_1v_2|b^i(\ell_1 uv_2-\ell_2uv_1)$. This gives that $v_1|b^i\ell_1uv_2$, and since $v_1$ is coprime to each of $b$, $\ell_1$, and $u$, we have that $v_1|v_2$. Similarly, $v_2|v_1$, and so $v_1=v_2$. In such a case, we have that $u^2v^2|b^iuv(\ell_1-\ell_2)$, and so $v|(\ell_1-\ell_2)$, so that $\ell_1\equiv \ell_2\ (\text{mod }v)$. This is possible for only $u$ number of residues $\ell$ modulo $uv$.

\end{proof}

Finally, we need decay at irrationals.

\begin{lemma}\label{lem:irrationalDecay}
    Suppose $\theta\in \mathbb{R}\setminus \mathbb{Q}$. Then, \begin{align*}
        \widehat{C}{b^N}(\theta)=o((b-s)^N).
    \end{align*}
\end{lemma}

\begin{proof}
    From the $L^\infty$ bound, one has that \begin{align*}
        |\widehat{C}_{b^N}(\theta)|\leq (b-s)^N\exp\Big(-\frac{1}{b}\sum_{i=0}^{N-1}\|b^i\theta\|^2\Big),
    \end{align*} and so it suffices to show that $\sum_{i=0}^{N-1}\|b^i\theta\|^2\rightarrow \infty$ as $N\rightarrow \infty$. Suppose by way of contradiction that $\sum_{i\geq 0}\|b^i\theta\|^2<\infty$, then necessarily one has that $\|b^i\theta\|\rightarrow 0$ as $i\rightarrow \infty$. In particular, there exists some $N$ such that for all $i\geq N$, $\|b^i\theta\|\leq b^{-2}$, say. Since $\theta$ is irrational, $\|b^N\theta\|\neq 0$, and so $0<\|b^N\theta\|\leq b^{-2}$. Then, $\|b^{N+1}\theta\|=b\|b^N\theta\|$, and by extension, for all $k\leq \lfloor \frac{\log (1/2\|b^N\theta\|)}{\log b}\rfloor $, one observes that \begin{align*}
        \|b^{N+k}\theta\|=b^k\|b^N\theta\|.
    \end{align*} At $k=\lfloor \frac{\log (1/2\|b^N\theta\|)}{\log b}\rfloor$, one may compute \begin{align*}
        \|b^{N+k}\theta\|=b^k\|b^N\theta\|\in \Big(\frac{1}{2b},\frac{1}{2}\Big],
    \end{align*} which is a contradiction.
\end{proof}

\begin{proof}[Proof of Theorem \ref{thm:VinogradovII}] Fix $\epsilon>0$ and take $\theta$ irrational. We will show that \begin{align*}
    \limsup_{N\rightarrow \infty}(b-s)^{-N}\Big|\sum_{n<b^N}\mathbf{1}_\mathcal{C}(n)\Lambda(n)e(\theta n)\Big|\ll_b \sup_{v>1/\epsilon}\frac{1}{\phi(v)}
\end{align*} Then, taking $\epsilon\rightarrow 0$, one has the result from the fact that the supremum tends to zero as $\epsilon\rightarrow 0$.\\

First, applying Proposition \ref{prop:ExponentialSumReduction} to the sum, for $A>0$ we may write \begin{align*}
    \sum_{n<b^N}\mathbf{1}_\mathcal{C}(n)\Lambda(n)e(\theta n)=\sum_{d<\log^{A'}(b^N)}\frac{\mu(d)}{\phi(d)}\sum_{\ell\in (\mathbb{Z}/d\mathbb{Z})^*}\widehat{C}_{b^N}\Big(\theta+\frac{\ell}{d}\Big)+O_A\Big(\frac{(b-s)^N}{\log^A(b^N)}\Big)
\end{align*} for some $A'>0$ depending on $A$. For each $d<\log^{A'}(b^N)$, we may write $d$ uniquely as $d=uv$ with $p|u\implies p|b$ and $(v,b)=1$. Since we may restrict $d$ to be squarefree, so may we restrict $u$, and so $u|b$. This gives that \begin{align*}
    \sum_{n<b^N}\mathbf{1}_\mathcal{C}(n)\Lambda(n)e(\theta n)=\sum_{u|b}\sum_{v<\log^{A'}(b^N)/u}\frac{\mu(uv)}{\phi(uv)}\sum_{\ell \in (\mathbb{Z}/uv\mathbb{Z})^*}\widehat{C}_{b^N}\Big(\theta+\frac{\ell}{uv}\Big)+O_A\Big(\frac{(b-s)^N}{\log^A(b^N)}\Big).
\end{align*} For each fixed $u|b$, Lemma \ref{lem:uncertainty} shows that there is at most one denominator $v<\log^{A'}(b^N)/u$ for which an exceptional residue exists. If such a denominator exists, denote it by $v_{u,N}$ and define $\mathcal{E}_{u,N}$ as follows: \begin{align*}
    \mathcal{E}_{u,N}:=\Big\{\ell\in (\mathbb{Z}/uv_{u,N}\mathbb{Z})^*:\Big|\widehat{C}_{b^N}\Big(\theta+\frac{\ell}{uv_{u,N}}\Big)\Big|>(b-s)^N\exp(-c_bN^{2/3})\Big\}
\end{align*} Otherwise, set $v_{u,N}=1$ and $\mathcal{E}_{u,N}=\varnothing$. Then, using the bound from Lemma \ref{lem:uncertainty} for the non-exceptional denominators, one obtains that \begin{align*}
    &\sum_{u|b}\sum_{v<\log^{A'}(b^N)/u}\frac{\mu(uv)}{\phi(uv)}\sum_{\ell\in (\mathbb{Z}/uv\mathbb{Z})^*}\widehat{C}_{b^N}\Big(\theta+\frac{\ell}{uv}\Big) \\ &=\sum_{u|b}\frac{\mu(uv_{u,N})}{\phi(uv_{u,N})}\sum_{\ell \in \mathcal{E}_{u,N}}\widehat{C}_{b^N}\Big(\theta+\frac{\ell}{uv_{u,N}}\Big)+O_b\Big((b-s)^N\log^{2A'}(b^N)\exp(-c_bN^{2/3})\Big).
\end{align*} We separate the main term into two parts: where $v_{u,N}\leq 1/\epsilon$ and where $v_{u,N}>1/\epsilon$: \begin{align*}
    \Sigma_1&:=\sum_{\substack{u|b\\ v_{u,N}\leq 1/\epsilon}}\frac{\mu(uv_{u,N})}{\phi(uv_{u,N})}\sum_{\ell \in \mathcal{E}_{u,N}}\widehat{C}_{b^N}\Big(\theta+\frac{\ell}{uv_{u,N}}\Big) \\ \Sigma_2&:=\sum_{\substack{u|b\\ v_{u,N}>1/\epsilon}}\frac{\mu(uv_{u,N})}{\phi(uv_{u,N})}\sum_{\ell\in \mathcal{E}_{u,N}}\widehat{C}_{b^N}\Big(\theta+\frac{\ell}{uv_{u,N}}\Big).
\end{align*} By the triangle inequality, \begin{align*}
    |\Sigma_1|\leq \sum_{\substack{u|b\\ v_{u,N}\leq 1/\epsilon}}\frac{1}{\phi(u)}\sum_{\ell \in \mathcal{E}_{u,N}}\Big|\widehat{C}_{b^N}\Big(\theta+\frac{\ell}{uv_{u,N}}\Big)\Big|
\end{align*} and since \begin{align*}\{\ell \in \mathcal{E}_{u,N}:v_{u,N}\leq 1/\epsilon\}\subseteq \{\ell: \ell\leq u\epsilon^{-1} \}\end{align*} one may majorize the sum above and produce that \begin{align*}
    |\Sigma_1|\leq \sum_{\substack{u|b\\ v_{u,N}\leq 1/\epsilon}}\frac{1}{\phi(u)}\sum_{\ell \leq u\epsilon^{-1}}\Big|\widehat{C}_{b^N}\Big(\theta+\frac{\ell}{uv_{u,N}}\Big)\Big|.
\end{align*} This is a linear combination of terms of the form $|\widehat{C}_{b^N}(\theta_i)|$ where each $\theta_i$ is irrational, and a finite number of terms that depend only on $\epsilon$ and $b$. By Lemma \ref{lem:irrationalDecay} we then deduce that this is $o((b-s)^N)$.\\

To bound $\Sigma_2$, one may use the trivial bound for $\widehat{C}$, and obtain \begin{align*}
    \Big|\sum_{\substack{u|b\\ v_{u,N}>1/\epsilon}}\frac{\mu(uv_{u,N})}{\phi(uv_{u,N})}\sum_{\ell\in \mathcal{E}_{u,N}}\widehat{C}_{b^N}\Big(\theta+\frac{\ell}{uv_{u,N}}\Big)\Big|&\leq (b-s)^N\sum_{\substack{u|b\\ v_{u,N}>1/\epsilon}}\frac{1}{\phi(uv_{u,N})}|\mathcal{E}_{u,N}| \\ &\leq (b-s)^N\sum_{\substack{u|b\\ v_{u,N}>1/\epsilon}}\frac{u}{\phi(uv_{u,N})} \\ &\ll_b (b-s)^N \sup_{v>1/\epsilon}\frac{1}{\phi(v)}.
\end{align*} Here, we used that $|\mathcal{E}_{u,N}|\leq u$. Taking $N\rightarrow \infty$, we obtain that \begin{align*}
    \limsup_{N\rightarrow \infty}(b-s)^{-N}\Big|\sum_{n<b^N}\mathbf{1}_\mathcal{C}(n)\Lambda(n)e(\theta n)\Big|\ll_b \sup_{v>1/\epsilon}\frac{1}{\phi(v)},
\end{align*} as desired.
    
\end{proof}

\section{Funding}

This work was supported by the Department of Education Graduate Assistance in Areas of National Need program at the Georgia Institute of Technology [P200A240169]; and the US National Science Foundation [DMS-2247254].

\section{Acknowledgements} The author thanks Michael Lacey for helpful feedback.

\printbibliography

@article{bourgainVDC,
  author  = {Bourgain, Jean},
  title   = {{Ruzsa}'s problem on sets of recurrence},
  journal = {Israel Journal of Mathematics},
  volume  = {59},
  number  = {2},
  pages   = {150--166},
  year    = {1987},
  doi     = {10.1007/BF02787258}
}

@article{bourgainPrimes1,
  author  = {Bourgain, Jean},
  title   = {Prescribing the binary digits of primes},
  journal = {Israel Journal of Mathematics},
  volume  = {194},
  number  = {2},
  pages   = {935--955},
  year    = {2013},
  doi     = {10.1007/s11856-012-0104-2}
}

@article{bourgainPrimes2,
  author  = {Bourgain, Jean},
  title   = {Prescribing the binary digits of primes, {II}},
  journal = {Israel Journal of Mathematics},
  volume  = {206},
  number  = {1},
  pages   = {165--182},
  year    = {2015},
  doi     = {10.1007/s11856-014-1129-5}
}

@misc{burgin2026integercantorsetsarithmetic,
  author        = {Burgin, Alex and Fragkos, Anastasios and
                   Lacey, Michael T. and Mena, Dario and
                   Reguera, Maria Carmen},
  title         = {Integer {Cantor} Sets: Arithmetic Combinatorial Properties},
  year          = {2026},
  eprint        = {2602.15292},
  archivePrefix = {arXiv},
  primaryClass  = {math.DS},
  url           = {https://arxiv.org/abs/2602.15292}
}

@misc{burgin2026szemeredistheoremcantorsets,
  author        = {Burgin, Alex and Fragkos, Anastasios and
                   Lacey, Michael T. and Mena, Dario and
                   Reguera, Maria Carmen},
  title         = {{Szemer{\'e}di}'s Theorem Along {Cantor} Sets of Integers},
  year          = {2026},
  eprint        = {2602.15299},
  archivePrefix = {arXiv},
  primaryClass  = {math.NT},
  url           = {https://arxiv.org/abs/2602.15299}
}

@article{semiprimes,
  author   = {Dartyge, C{\'e}cile and Mauduit, Christian},
  title    = {Ensembles de densit{\'e} nulle contenant des entiers
              poss{\'e}dant au plus deux facteurs premiers},
  journal  = {Journal of Number Theory},
  volume   = {91},
  number   = {2},
  pages    = {230--255},
  year     = {2001},
  doi      = {10.1006/jnth.2001.2681},
  language = {French}
}

@article{EMS1,
  author  = {Erd{\H{o}}s, Paul and Mauduit, Christian and
             S{\'a}rk{\"o}zy, Andr{\'a}s},
  title   = {On arithmetic properties of integers with missing digits.
             {I}. Distribution in residue classes},
  journal = {Journal of Number Theory},
  volume  = {70},
  number  = {2},
  pages   = {99--120},
  year    = {1998},
  doi     = {10.1006/jnth.1998.2229}
}

@article{EMS2,
  author  = {Erd{\H{o}}s, Paul and Mauduit, Christian and
             S{\'a}rk{\"o}zy, Andr{\'a}s},
  title   = {On arithmetic properties of integers with missing digits.
             {II}. Prime factors},
  journal = {Discrete Mathematics},
  volume  = {200},
  number  = {1--3},
  pages   = {149--164},
  year    = {1999},
  doi     = {10.1016/S0012-365X(98)00331-8},
  note    = {Paul Erd{\H{o}}s memorial collection}
}

@article{furstenbergKatznelson,
  author  = {Furstenberg, Hillel and Katznelson, Yitzhak},
  title   = {An ergodic {Szemer{\'e}di} theorem for {IP}-systems
             and combinatorial theory},
  journal = {Journal d'Analyse Math{\'e}matique},
  volume  = {45},
  number  = {1},
  pages   = {117--168},
  year    = {1985},
  doi     = {10.1007/BF02792547}
}

@article{greenPowerSavings,
  author  = {Green, Ben},
  title   = {On {S{\'a}rk{\"o}zy}'s theorem for shifted primes},
  journal = {Journal of the American Mathematical Society},
  volume  = {37},
  number  = {4},
  pages   = {1121--1201},
  year    = {2024},
  doi     = {10.1090/jams/1036}
}

@article{greenTao,
  author  = {Green, Ben and Tao, Terence},
  title   = {Linear equations in the primes},
  journal = {Annals of Mathematics},
  series  = {2},
  volume  = {171},
  number  = {3},
  pages   = {1753--1850},
  year    = {2010},
  doi     = {10.4007/annals.2010.171.1753}
}

@article{greenTaoMobius,
  author  = {Green, Ben and Tao, Terence},
  title   = {The {M{\"o}bius} function is strongly orthogonal to nilsequences},
  journal = {Annals of Mathematics},
  series  = {2},
  volume  = {175},
  number  = {2},
  pages   = {541--566},
  year    = {2012},
  doi     = {10.4007/annals.2012.175.2.3}
}

@article{greenTaoZiegler,
  author  = {Green, Ben and Tao, Terence and Ziegler, Tamar},
  title   = {An inverse theorem for the {Gowers}
             {$U^{s+1}[N]$}-norm},
  journal = {Annals of Mathematics},
  series  = {2},
  volume  = {176},
  number  = {2},
  pages   = {1231--1372},
  year    = {2012},
  doi     = {10.4007/annals.2012.176.2.11}
}

@book{iwaniecKowalski,
  author    = {Iwaniec, Henryk and Kowalski, Emmanuel},
  title     = {Analytic Number Theory},
  series    = {American Mathematical Society Colloquium Publications},
  volume    = {53},
  publisher = {American Mathematical Society},
  address   = {Providence, RI},
  year      = {2004},
  doi       = {10.1090/coll/053},
  isbn      = {978-0-8218-3633-0}
}

@article{kamaeMendesFrance,
  author  = {Kamae, T. and Mend{\`e}s France, Michel},
  title   = {{Van der Corput}'s difference theorem},
  journal = {Israel Journal of Mathematics},
  volume  = {31},
  number  = {3--4},
  pages   = {335--342},
  year    = {1978},
  doi     = {10.1007/BF02761498}
}

@article{konyagin,
  author  = {Konyagin, Sergei V.},
  title   = {Arithmetic properties of integers with missing digits:
             distribution in residue classes},
  journal = {Periodica Mathematica Hungarica},
  volume  = {42},
  number  = {1--2},
  pages   = {145--162},
  year    = {2001},
  doi     = {10.1023/A:1015256809636}
}

@article{lengSawhney,
  author  = {Leng, James and Sawhney, Mehtaab},
  title   = {{Vinogradov}'s Theorem for Primes With Restricted Digits},
  journal = {International Mathematics Research Notices},
  volume  = {2025},
  number  = {3},
  pages   = {rnae294},
  year    = {2025},
  doi     = {10.1093/imrn/rnae294}
}

@article{lucier,
  author  = {Lucier, Jason},
  title   = {Difference sets and shifted primes},
  journal = {Acta Mathematica Hungarica},
  volume  = {120},
  number  = {1--2},
  pages   = {79--102},
  year    = {2008},
  doi     = {10.1007/s10474-007-7107-1}
}

@article{maynardDecimal,
  author  = {Maynard, James},
  title   = {Primes with restricted digits},
  journal = {Inventiones Mathematicae},
  volume  = {217},
  number  = {1},
  pages   = {127--218},
  year    = {2019},
  doi     = {10.1007/s00222-019-00865-6}
}

@article{maynardAsymptotic,
  author  = {Maynard, James},
  title   = {Primes and polynomials with restricted digits},
  journal = {International Mathematics Research Notices},
  volume  = {2022},
  number  = {14},
  pages   = {10626--10648},
  year    = {2022},
  doi     = {10.1093/imrn/rnab002}
}

@article{nathBV,
  author  = {Nath, Kunjakanan},
  title   = {Primes with a missing digit: Distribution in arithmetic
             progressions and an application in sieve theory},
  journal = {Journal of the London Mathematical Society},
  series  = {2},
  volume  = {109},
  number  = {1},
  pages   = {e12837},
  year    = {2024},
  doi     = {10.1112/jlms.12837}
}

@article{ruzsaSanders,
  author  = {Ruzsa, Imre Z. and Sanders, Tom},
  title   = {Difference sets and the primes},
  journal = {Acta Arithmetica},
  volume  = {131},
  number  = {3},
  pages   = {281--301},
  year    = {2008},
  doi     = {10.4064/aa131-3-5}
}

@article{saavedraAraya,
  author  = {Saavedra-Araya, Vicente},
  title   = {Distribution of integers with digit restrictions via
             {Markov} chains},
  journal = {Ergodic Theory and Dynamical Systems},
  volume  = {46},
  number  = {3},
  pages   = {757--804},
  year    = {2026},
  doi     = {10.1017/etds.2025.10256}
}

@article{sarkozySquares,
  author  = {S{\'a}rk{\"o}zy, Andr{\'a}s},
  title   = {On difference sets of sequences of integers. {I}},
  journal = {Acta Mathematica Academiae Scientiarum Hungaricae},
  volume  = {31},
  number  = {1--2},
  pages   = {125--149},
  year    = {1978},
  doi     = {10.1007/BF01896079}
}

@article{sarkozyOriginal,
  author  = {S{\'a}rk{\"o}zy, Andr{\'a}s},
  title   = {On difference sets of sequences of integers. {III}},
  journal = {Acta Mathematica Academiae Scientiarum Hungaricae},
  volume  = {31},
  number  = {3--4},
  pages   = {355--386},
  year    = {1978},
  doi     = {10.1007/BF01901984}
}

@article{primesQuant,
  author  = {Slijep{\v{c}}evi{\'c}, Sini{\v{s}}a},
  title   = {On van der {Corput} property of shifted primes},
  journal = {Functiones et Approximatio Commentarii Mathematici},
  volume  = {48},
  number  = {1},
  pages   = {37--50},
  year    = {2013},
  doi     = {10.7169/facm/2013.48.1.4}
}

@article{swaenepoel,
  author  = {Swaenepoel, Cathy},
  title   = {Prime numbers with a positive proportion of preassigned digits},
  journal = {Proceedings of the London Mathematical Society},
  series  = {3},
  volume  = {121},
  number  = {1},
  pages   = {83--151},
  year    = {2020},
  doi     = {10.1112/plms.12314}
}

@article{wang,
  author  = {Wang, Ruoyi},
  title   = {On a theorem of {S{\'a}rk{\"o}zy} for difference sets
             and shifted primes},
  journal = {Journal of Number Theory},
  volume  = {211},
  pages   = {220--234},
  year    = {2020},
  doi     = {10.1016/j.jnt.2019.10.009}
}

@misc{brokeringIosevichKrause,
  author        = {{Brokering Pinilla}, F{\'e}lix and
                   Iosevich, Alex and
                   Krause, Ben},
  title         = {Pointwise Convergence of Ergodic Averages Along
                   Integer {Cantor} Sets},
  year          = {2026},
  eprint        = {2607.16064},
  archivePrefix = {arXiv},
  primaryClass  = {math.DS},
  url           = {https://arxiv.org/abs/2607.16064}
}
\end{document}